\documentclass[11pt]{amsart}
\usepackage{amsmath}
\usepackage{amsfonts}
\usepackage{amssymb}
\usepackage{mathtools}
\usepackage{mathrsfs}
\usepackage{graphicx}
\usepackage{epstopdf}
\usepackage{alias cnt}
\usepackage{overpic}
\usepackage{enumitem,times}
\usepackage[bookmarks=true,pdfstartview=FitH, pdfborder={0 0 0}, colorlinks=true,citecolor=blue, linkcolor=blue]{hyperref}

\newtheorem{theorem}{Theorem}[section]
\newaliascnt{conj}{theorem}
\newaliascnt{cor}{theorem}
\newaliascnt{lemma}{theorem}
\newaliascnt{fact}{theorem}
\newaliascnt{claim}{theorem}
\newaliascnt{prop}{theorem}
\newaliascnt{definition}{theorem}
\newaliascnt{assump}{theorem}
\newaliascnt{question}{theorem}
\newaliascnt{convention}{theorem}

\newtheorem{lemma}[lemma]{Lemma}

\newtheorem{prop}[prop]{Proposition}
\newtheorem{definition}[definition]{Definition}

\newtheorem{assump}[assump]{Assumption}

\numberwithin{theorem}{section}
\numberwithin{figure}{section}

\aliascntresetthe{conj} \aliascntresetthe{cor}
\aliascntresetthe{lemma} \aliascntresetthe{fact}
\aliascntresetthe{claim} \aliascntresetthe{prop}
\aliascntresetthe{definition} \aliascntresetthe{question}
\aliascntresetthe{assump} \aliascntresetthe{convention}

\makeatletter
\let\oldtheequation\theequation
\renewcommand\tagform@[1]{\maketag@@@{\ignorespaces#1\unskip\@@italiccorr}}
\renewcommand\theequation{(\oldtheequation)}
\makeatother

\theoremstyle{remark}
\newaliascnt{rmk}{theorem}
\newtheorem{remark}[rmk]{Remark}
\aliascntresetthe{rmk}

\theoremstyle{remark}
\newaliascnt{exam}{theorem}
\newtheorem{exam}[exam]{Example}
\aliascntresetthe{exam}

\def\sek~{\S{}}

\makeatletter
\newcommand{\mylabel}[2]{#2\def\@currentlabel{#2}\label{#1}}

\DeclareFontFamily{U}{mathx}{\hyphenchar\font45}
\DeclareFontShape{U}{mathx}{m}{n}{
	<5> <6> <7> <8> <9> <10>
	<10.95> <12> <14.4> <17.28> <20.74> <24.88>
	mathx10
}{}
\DeclareSymbolFont{mathx}{U}{mathx}{m}{n}
\DeclareFontSubstitution{U}{mathx}{m}{n}
\DeclareMathAccent{\widecheck}{0}{mathx}{"71}
\DeclareMathAccent{\wideparen}{0}{mathx}{"75}

\DeclarePairedDelimiter\abs{\lvert}{\rvert}
\DeclarePairedDelimiter\norm{\lVert}{\rVert}

\DeclareMathOperator{\std}{std}

\DeclareMathOperator{\ind}{ind}
\DeclareMathOperator{\link}{Lk}

\DeclareMathOperator{\supp}{Supp}
\DeclareMathOperator{\tb}{tb}

\newcommand{\Cbb}{\mathbb{C}}
\newcommand{\Rbb}{\mathbb{R}}
\newcommand{\Zbb}{\mathbb{Z}}

\newcommand{\pbf}{\mathbf{p}}
\newcommand{\qbf}{\mathbf{q}}

\newcommand{\xbf}{\mathbf{x}}

\newcommand{\Fcal}{\mathcal{F}}
\newcommand{\Gcal}{\mathcal{G}}

\newcommand{\Ocal}{\mathcal{O}}

\newcommand{\open}{\mathcal{O}p}

\newcommand{\s}{\vskip.1in}
\newcommand{\n}{\noindent}
\newcommand{\p}{\partial}

\newcommand{\be}{\begin{enumerate}}
	\newcommand{\ee}{\end{enumerate}}
\newcommand{\op}{\operatorname}

\begin{document}

\title{Existence of coLegendrians in contact $5$-manifolds}

%\author{Ko Honda}
%\address{University of California, Los Angeles, Los Angeles, CA 90095}
%\email{honda@math.ucla.edu} \urladdr{http://www.math.ucla.edu/\char126 honda}

\author{Yang Huang}
%\address{University of Southern Denmark}
%\email{hymath@gmail.com} \urladdr{https://sites.google.com/site/yhuangmath}

%\date{This version: \today}

\thanks{YH is supported by the Centre for Quantum Mathematics at the University of Southern Denmark.}

\begin{abstract}
We introduce the notion of coLegendrian submanifold in a contact manifold as a middle-dimensional coisotropic submanifold and study its existence in dimension $5$ via contact Morse theory. As an application, the second generation of Orange is now available.
\end{abstract}

\maketitle

\tableofcontents

\section{Introduction}
The study of embedded surfaces is indispensable in $3$-manifold topology and has a history of more than $100$ years. For the relatively new interests in understanding contact structures on $3$-manifolds, one finds it important, not surprisingly, to study embedded surfaces together with germs of contact structures on them. As a relative version, Legendrian knots are studied, with great success, by examining the germ of contact structure on spanning surfaces, e.g., Seifert surfaces. This tradition goes back to the work of Bennequin \cite{Ben83}, Eliashberg \cite{Eli89,Eli92}, and Fuchs-Tabachnikov\cite{FT97} at the very beginning of the subject, followed by more systematic development by Giroux \cite{Gi91,Gi00}, Honda \cite{Hon00}, Etnyre-Honda \cite{EH03} and many others. As of today, $3$-dimensional contact topology is fairly well understood, at least in comparison with the higher dimensional case.

The focus of this paper is on contact manifolds $(M,\xi)$ of dimension $2n+1>3$. In this case, one can develop a parallel theory for hypersurfaces $\Sigma^{2n} \subset M$ generalizing the study of surfaces in contact $3$-manifolds. This is carried out in \cite{HH18,HH19}, where a \emph{contact Morse theory} is established in the spirit of Eliashberg-Gromov \cite{EG91} and refining Giroux's work \cite{Gi91,Gi00} in dimension $3$ and \cite{Gi02} in higher dimensions. However, the corresponding relative theory for Legendrian submanifolds $\Lambda \subset M$ breaks down for dimensional reasons, i.e., $\Lambda$ cannot be the boundary of a hypersurface unless $\dim M=3$. As a piece of jargon, we say a compact submanifold $Y \subset M$ is a \emph{filling} of $\Lambda$ if $\Lambda=\p Y$. All (sub-)manifolds in this paper will be assumed to be orientable.

One motivation for studying fillings of $\Lambda$ (and the associated contact germ) comes from the desire of understanding isotopies of Legendrians. Suppose $\Lambda_t, t \in [0,1]$, is a Legendrian isotopy. Information of this isotopy is captured by the totality $Y \coloneqq \cup_{t \in [0,1]} \Lambda_t$, which we pretend to be a submanifold although it is often not. Then $Y$ is a filling of $\Lambda_0 \cup \Lambda_1$ and it is foliated by Legendrians. The above consideration yields two important, but loosely formulated, observations: first, if one wants to know whether two Legendrians $\Lambda_0,\Lambda_1$ are Legendrian isotopic (given they are topologically isotopic), then one can look for a (topologically trivial) filling of $\Lambda_0 \cup \Lambda_1$ and try to normalize the contact germ on the filling so that it looks like the totality of a Legendrian isotopy; second, a submanifold smoothly foliated by Legendrian leaves is a special case of the so-called \emph{coisotropic submanifolds} introduced in \cite{Hua15}. We recall the definition of coisotropic submanifolds and introduce the key object of this paper: \emph{coLegendrian submanifolds}, as follows.

\begin{definition} \label{defn:coLeg submfd}
A smoothly embedded submanifold $Y \subset (M,\xi)$ is \emph{coisotropic} if $T_x Y \cap \xi(x) \subset \xi(x)$ is a coisotropic subspace with respect to the canonical conformal symplectic structure for every $x \in Y$. A coisotropic $Y$ is a \emph{coLegendrian} if $\dim Y=n+1$, given $\dim M=2n+1$.
\end{definition}

CoLegendrians (under a different name) and the associated Legendrian foliations were studied in \cite{Hua15,Hua14} with little success for one reason: they are extremely difficult to find in a given contact manifold. Indeed, Gromov's $h$-principle techniques from \cite{Gro86} can, at best, be used to show the existence of \emph{immersed} coLegendrians and not the embedded ones. Nonetheless, under different disguises, coLegendrians have appeared in the literature many times, mostly for a completely different reason: to detect the tight--overtwisted dichotomy of contact structures due to Eliashberg \cite{Eli89} in dimension $3$ and Borman-Eliashberg-Murphy \cite{BEM15} in higher dimensions. See, for example, Niederkr\"uger \cite{Nie06}, Massot-Niederkr\"uger-Wendl \cite{MNW13} and \cite{HH18}.

It is the purpose of this paper to show the abundance of (embedded) coLegendrians with certain singularities, which will be explained later, in any contact $5$-manifold using contact Morse theory. In the rest of the introduction, we will sketch the main ideas of this paper, leaving technical details to subsequent sections.

Recall from \cite{HH19} that a hypersurface $\Sigma \subset (M,\xi)$ is \emph{Morse} if the characteristic foliation $\Sigma_{\xi}$ is a Morse vector field, i.e., it is gradient-like for some Morse function on $\Sigma$. Furthermore, we say $\Sigma$ is Morse$^+$ if, in addition, there exists no flow line from a negative critical point of $\Sigma_{\xi}$ to a positive one. \emph{In this paper, (Morse) singularities of characteristic foliations will always be called critical points even without explicit mention of the Morse function.} By analogy with the definition of regular Lagrangians in the theory of Weinstein manifolds by Eliashberg-Ganatra-Lazarev \cite{EGL20}, a Legendrian $\Lambda \subset \Sigma$ is said to be \emph{regular} if it is tangent to $\Sigma_{\xi}$. Similarly, a coLegendrian $Y \subset \Sigma$ is \emph{regular} if it is tangent to $\Sigma_{\xi}$. The key point here is that regular (co)Legendrians inherits a Morse function, or equivalently a handle decomposition, from $\Sigma$.

Let's assume $\dim M=5$ for the rest of the introduction although many of our discussions can be easily generalized to any dimension. We now describe the singularities of a regular coLegendrian $Y$ in order for it to exist in abundance. Suppose $Y$ passes through an index $0$ critical point $p_0 \in \Sigma$. A neighborhood of $p_0$ in $\Sigma$ can be identified with the standard symplectic $(B^4,\omega_{\std})$ such that $p_0=0$ and the characteristic foliation is identified with the Liouville vector field $X=\tfrac{1}{2}r\p_r$. The \emph{link} of $p_0$ is the standard contact $3$-sphere $\p B^4$. Then $\link(Y,p_0) \coloneqq Y \cap \p B^4$ is a (smooth) $2$-sphere in $(S^3,\xi_{\std})$. We say $p_0$ is a \emph{cone singularity} of $Y$. Note that topologically speaking, such a cone singularity can always be smoothed. However, geometrically speaking, with respect to the Euclidean metric on $B^4$, the cone singularity is smooth precisely when $\link(Y,p_0)$ is a great $S^2$ in the round $S^3$. Contact topology stands in between topology and geometry, so does the cone singularity. Namely, an $S^2 \subset (S^3,\xi_{\std})$ is said to be \emph{standard} if its characteristic foliation has exactly one source and one sink and every flow line goes from the source to the sink. For example, if we identify $B^4 \subset \Cbb^2$ with the unit ball such that the contact structure on $\p B^4$ is given by the tangential complex lines. Then one can check that every great $S^2 \subset \p B^4$ is standard.

We say $p_0$ is \emph{smooth(able)} if $\link(Y,p_0)$ is standard. If $p_0$ is smoothable, then indeed it can be smoothed, while staying regular, by a small perturbation of $\Sigma$. Similarly, cone singularities can be defined for index $4$ critical points $p_4 \in \Sigma$ by simply reversing the direction of the characteristic foliation.

With the above preparation, we can state the following existence theorem of regular coLegendrians.

\begin{theorem} \label{thm:coLeg approx}
Suppose $(M,\xi)$ is contact $5$-manifold. Then any closed $3$-submanifold $Y$ with trivial normal bundle can be $C^0$-approximated by a regular coLegendrian with isolated cone singularities. If $Y$ is compact with (smooth) Legendrian boundary, then the same approximation holds such that $\p Y$ is a regular Legendrian. In particular, the cone singularities stay away from $\p Y$.
\end{theorem}

\begin{remark}
A technical, but important, point in our approach to (co)Legendrian submanifolds in this paper is that all submanifolds are \emph{unparameterized}, i.e., they are subsets of the ambient manifold rather than (smooth) embeddings. This is to be compared with \cite{Mur12}, where all Legendrians (singular or smooth) are parameterized.
\end{remark}

Note that the triviality of $T_Y M$ is necessary since the definition of regularity requires $Y$ to be contained in a hypersurface (here we use $\dim M=5$). In fact, \autoref{thm:coLeg approx} follows readily from \cite[Section 12]{HH19} without the conditions on the singularities of $Y$, and the main contribution of this paper is to eliminate all the singularities except possibly the cones. The proof of \autoref{thm:coLeg approx} will be carried out in \autoref{sec:regular coLeg}, after we set up the background theory of regular Legendrians in \autoref{sec:regular Leg}.

Let's point out that \autoref{thm:coLeg approx} by itself is rather useless in improving our understanding of either contact structures or Legendrians. Instead, it is through the proof of the theorem that we understand how regular coLegendrians are built out of handles and how such handles can be manipulated. However, since our main focus of this paper is just to establish the existence of coLegendrians, we postpone a more thorough study of coLegendrian handle decomposition to a forthcoming work where the theory of coLegendrians will be applied to study isotopies of Legendrians.

As an application of our study of coLegendrians, we introduce in \autoref{sec:orange II} the second generation of \emph{overtwisted orange} (cf. \cite{HH18}), which we call \emph{orange II} and denoted by $\Ocal_2$. Although we construct $\Ocal_2$ as a regular coLegendrian, it is the Legendrian foliation $\Fcal_{\Ocal_2}$ which determines the (overtwisted) contact germ. Hence, in fact, one can define $\Ocal_2$ without any reference to regular coLegendrians, and \autoref{sec:orange II} can be understood independent of the rest of the paper. The second generation orange is considerably simpler than the first and maybe also other known overtwisted objects. But the road leading to its discovery is rougher than one might think.

We wrap up the introduction with an explanation of our choice in this paper of giving up the term ``convex hypersurface theory'' used in \cite{HH19}.

\begin{remark}
In the foundational papers \cite{Gi91,Gi00}, Giroux introduced the notion of \emph{convex hypersurfaces} into contact topology, which has then been used ubiquitously especially in the study contact $3$-manifolds and also in \cite{HH18,HH19}. A hypersurface is \emph{convex} in the sense of Giroux if it is transverse to a contact vector field. A key feature of convex surfaces in dimension $3$ is that the so-called \emph{dividing set} completely determines the contact germ on the surface, reducing $3$-dimensional contact topology essentially to a combinatorial problem. Such useful feature is, unfortunately, not available in higher dimensions. Indeed, the main contribution of \cite{HH19} is to show that $\Sigma_{\xi}$ can be made Morse by a $C^0$-small perturbation of $\Sigma$, rather than showing $\Sigma$ can be made convex, even though Morse$^+$ hypersurfaces are convex as it turns out. As we will see in this paper, many hypersurfaces that are important in our study of contact structures are \emph{not} convex, but only Morse, even in dimension $3$. Due to this change of perspective, we will drop Giroux's convexity from our terminology and rely instead on Morse-theoretic notions. However, certain convenient terminologies, such as the dividing set, will be retained.
\end{remark}

\section{Basics on contact Morse theory} \label{sec:contact Morse theory}
In this section, we recall the contact Morse theory on hypersurfaces established in \cite{HH19}, which is also the starting point of this paper.

Let $(M,\xi)$ be a contact manifold of dimension $2n+1$. Following \cite{HH19}, a hypersurface $\Sigma \subset M$ is \emph{Morse} if the characteristic foliation $\Sigma_{\xi}$, viewed as a vector field, is Morse, i.e., there exists a Morse function $f: \Sigma \to \Rbb$ such that $\Sigma_{\xi}$ is gradient-like with respect to $f$. A Morse hypersurface $\Sigma$ is Morse$^+$ if, in addition, there exists no flow lines of $\Sigma_{\xi}$ from a negative critical point to a positive one. Suppose $\Sigma$ is Morse$^+$. Define the \emph{dividing set} $\Gamma \subset \Sigma$ to be the boundary of the handlebody built by the positive handles. Clearly $\Gamma$ is well-defined up to isotopy. Moreover $\Sigma \setminus \Gamma$ is naturally the disjoint union of two Weinstein manifolds. Following \cite{Gi91}, write $\Sigma \setminus \Gamma=R_+ \cup R_-$ such that $R_{\pm}$ is the Weinstein manifold built by the positive/negative handles, respectively.

\begin{remark}
In this paper, we will not in general assume that $\Sigma$ is closed, and \emph{no} boundary condition will be imposed in general, e.g., $\p \Sigma$ needs not to be transverse or tangent to $\Sigma_{\xi}$. Of course, in this case, one cannot decompose $\Sigma$ into Weinstein manifolds as in the closed case.
\end{remark}

%It is often convenient to allow more general singularities than the nondegenerate ones when families of hypersurfaces are under consideration. The following convention specifies the allowable singularities.

%\begin{convention} \label{conv:morse in family}
%In addition to nondegenerate singularities, one is allowed to have birth-death type singularities for $1$-parameter families of $\Sigma_{\xi}$, and furthermore swallowtail singularities for $2$-parameter families of $\Sigma_{\xi}$.
%\end{convention}

The following result from \cite{HH19} shows that the assumption of $\Sigma_{\xi}$ being Morse is rather mild.

\begin{theorem} \label{thm:morse is generic}
Any hypersurface can be $C^0$-approximated by a Morse hypersurface, which can be further $C^{\infty}$-perturbed to becomes Morse$^+$.
\end{theorem}

Indeed, \autoref{thm:morse is generic} can be extended to a relative version. Namely, if there exists a closed subset $K \subset \Sigma$ such that $\Sigma_{\xi}$ is already Morse on an open neighborhood of $K$, then $\Sigma$ be $C^0$-approximated by a Morse hypersurface relative to $K$.

\section{Regular Legendrians} \label{sec:regular Leg}
In this section, we apply the contact Morse theory introduced in \autoref{sec:contact Morse theory} to Legendrian submanifolds. We work with arbitrary dimensions in this section since there is nothing special about the theory of regular Legendrians in dimension $5$, in contrast to the theory of coLegendrians to be discussed in \autoref{sec:regular coLeg}.

Inspired by the notion of \emph{regular Lagrangians} in Weinstein manifolds introduced by Eliashberg-Ganatra-Lazarev \cite{EGL20}, we define \emph{regular Legendrians} as follows. 

\begin{definition} \label{defn:regular legendrian}
A Legendrian $\Lambda \subset \Sigma \subset (M,\xi)$ is \emph{regular} with respect to a Morse hypersurface $\Sigma$ if it is tangent to $\Sigma_{\xi}$, and nondegenerate critical points of $\Sigma_{\xi}$ on $\Lambda$ restrict to nondegenerate critical points of $\Sigma_{\xi}|_{\Lambda}$ in $\Lambda$.
\end{definition}

%\begin{figure}[ht]
%	\begin{overpic}[scale=.5]{quasi_regular.eps}
%	\end{overpic}
%	\caption{The blue lines depict a quasi-regular Legendrian on the left and a regular Legendrian on the right, respectively.}
%	\label{fig:quasi regular}
%\end{figure}

It is a difficult problem (cf. \cite[Problem 2.5]{EGL20}) in symplectic topology to find non-regular Lagrangians in a Weinstein manifold. It turns out that the contact topological counterpart is much more flexible. Moreover, observe that the very definition of regular Legendrians depends on a choice of the hypersurface $\Sigma \supset \Lambda$. We will familiarize ourselves with regular Legendrians through the following examples.

\begin{exam} \label{ex:collar nbhd}
Any Legendrian $\Lambda$ is regular with respect to the hypersurface $\Sigma \coloneqq T^{\ast} \Lambda \subset (M,\xi)$ such that $\Lambda \subset T^{\ast} \Lambda$ as the 0-section. Indeed, by Legendrian neighborhood theorem, there exists a tubular neighborhood $(U(\Lambda), \xi|_{U(\Lambda)})$ of $\Lambda$ which is contactomorphic to a tubular neighborhood of the 0-section in $J^1(\Lambda) = \Rbb_z \times T^{\ast} \Lambda$, equipped with the standard contact structure. It remains to Morsify the canonical Liouville form $pdq$ on $T^{\ast} \Lambda$  (cf. \cite[Example 11.12]{CE12}) so that $\Sigma$ is Morse and $\Lambda$ is tangent to $\Sigma_{\xi}$.
\qed
\end{exam}

\begin{exam} \label{ex:funny unknot}
This example is the reformulation of a result of Courte-Ekholm in \cite{CE18}. Let $\Sigma \coloneqq S^{2n} \subset (\Rbb^{2n+1}, \xi_{\std})$ be the unit sphere, which is clearly Morse$^+$. Moreover, in the decomposition $\Sigma \setminus \Gamma = R_+ \cup R_-$, the dividing set $\Gamma$ is contactomorphic to $(S^{2n-1}, \eta_{\std})$, and $R_{\pm}$ are both symplectomorphic to the standard symplectic vector space $(\Rbb^{2n}, \omega_{\std})$.
	
Let $(B^{2n}, \omega_{\std})$ be the standard symplectic filling of $(S^{2n-1}, \eta_{\std})$. Suppose $\Lambda_0$ is a Legendrian sphere in $S^{2n-1}$ which bounds a regular Lagrangian disk $D \subset B^{2n}$. Identify $\Gamma=S^{2n-1}$ and take two copies $D_{\pm} \subset R_{\pm}$ of $D$, modulo obvious completions, respectively.
	
Define the Legendrian sphere $\Lambda \coloneqq D_+ \cup_{\Lambda_0} D_- \subset \Sigma$. Then $\Lambda$ is regular by construction. Moreover, it follows from \cite{CE18} that $\Lambda$ is in fact Legendrian isotopic to the standard Legendrian unknot, regardless of the choices of $D$ and $\Lambda_0$.
\qed
\end{exam}

\begin{exam} \label{ex:knot in 3d}
We restrict our attention to contact 3-manifolds in this example. Suppose $\Sigma \subset (M^3,\xi)$ is a Morse$^+$ surface, and $\Lambda \subset \Sigma$ is a regular Legendrian loop. Then $\Lambda$ is transverse to the dividing set $\Gamma$. Let $\abs{\Lambda \cap \Gamma}$ be the (honest) count of intersection points, which turns out to be always even.
	
Recall that given a framing $\sigma$ on $\Lambda$, i.e., a trivialization of the normal bundle $T_{\Lambda} M$, the \emph{Thurston-Bennequin invariant} $\tb_{\sigma} (\Lambda) \in \Zbb$ measures the total rotation of $\xi$ along $\Lambda$. See e.g. \cite{Et05} for more details.  Now note that $\Sigma$ uniquely specifies a framing on $\Lambda$, with respect to which we have 
\begin{equation} \label{eqn:tb computation}
\tb_{\Sigma} (\Lambda) = - \tfrac{1}{2}\abs{\Lambda \cap \Gamma}
\end{equation}
	
Suppose $(M,\xi)=(\Rbb^3,\xi_{\std})$. Then one can take $\Sigma$ to be any Seifert surface of $\Lambda$, and $\tb(\Lambda) = \tb_{\Sigma} (\Lambda)$ is independent of the choice of $\Sigma$. It follows from \autoref{eqn:tb computation} that $\Sigma$ can be made Morse$^+$ only if $\tb(\Lambda) \leq 0$. The other classical invariant: the \emph{rotation number} $r(\Lambda)$ (cf. \cite{Et05} for the definition) can also be computed in this setup by
\begin{equation} \label{eqn:rotation number}
r(\Lambda) = \chi(R_+) - \chi(R_-),
\end{equation}
where $\Sigma \setminus \Gamma = R_+ \cup R_-$ is the usual decomposition. Note, however, that $r(\Lambda)$ depends on an orientation of $\Lambda$, which by convention is the induced orientation from an orientation on $\Sigma$.

In the proof of \autoref{prop:regular Leg}, we will generalize \autoref{eqn:tb computation} to Morse surfaces and to higher dimensions (cf. \autoref{eqn:compute 1-framing}). The higher-dimensional counterpart of \autoref{eqn:rotation number} is \autoref{lem:Leg self intersection number}.
\qed
\end{exam}

We highlight a few special features about (regular) Legendrian knots from \autoref{ex:knot in 3d} which, as we will see, are in sharp contrast to the higher-dimensional case. Firstly, a Legendrian $\Lambda$ can be realized as a regular Legendrian in a Morse$^+$ surface $\Sigma$ only if $\tb_{\Sigma}(\Lambda) \leq 0$. Secondly, if $\Sigma$ is a Seifert surface, then the intersection $\Lambda \cap \Gamma$, as a finite set of points, is an invariant of $\Lambda$, i.e., the Thurston-Bennequin invariant. This is not true in higher dimensions, i.e., the topology of $\Lambda \cap \Gamma$ is \emph{not} an invariant of $\Lambda$.

Now we turn to the problem of realizing any Legendrian as a regular Legendrian in a given hypersurface. Let $\Lambda \subset (M,\xi)$ be a $n$-dimensional closed Legendrian. As usual, identify a tubular neighborhood of $\Lambda$ with a tubular neighborhood of the $0$-section in $J^1(\Lambda) = \Rbb_z \times T^{\ast} \Lambda$. Fix a Riemannian metric $g$ on $M$. Define the \emph{normal sphere bundle} 
\begin{equation*}
S_{\Lambda} M \coloneqq \{ v \in T_{\Lambda} M ~|~ \norm{v}_g=1 \}.
\end{equation*}
For a suitable choice of $g$, we can assume $S_{\Lambda} M \subset J^1(\Lambda)$ and $\p_z \in S_{\Lambda} M$.

\begin{definition} \label{defn:1-framing}
A \emph{1-framing} of $\Lambda$ is a section of $S_{\Lambda} M \to \Lambda$. The 1-framing defined by $\p_z$ is called the \emph{canonical 1-framing} of $\Lambda$.
\end{definition}

Up to homotopy, the canonical 1-framing is specified by any unit vector field along $\Lambda$ which is positively transverse to $\xi$. The set of 1-framings of a given Legendrian is described by the following lemma.

\begin{lemma} \label{lem:set of 1-framings}
The homotopy classes of $1$-framings of $\Lambda$ can be canonically identified with $\Zbb$ such that the canonical $1$-framing is identified with zero.
\end{lemma}

\begin{proof}
This is a standard consequence of the Pontryagin-Thom construction (cf. \cite[Chapter 7]{Mil65}). I learned the following trick from Patrick Massot around 2012. Given two sections $\sigma_i: \Lambda \to S_{\Lambda} M,i =0,1$, we consider the following set
\begin{equation*}
L \coloneqq \{ x \in \Lambda ~|~ \sigma_0(x) = -\sigma_1(x) \}.
\end{equation*}
Generically $L$ is an oriented 0-dimensional compact submanifold of $\Lambda$. The signed count of points in $L$ defines the difference between $\sigma_i, i=0,1$. Conversely, given a 1-framing $\sigma: \Lambda \to S_{\Lambda} M$ and an integer $k \in \Zbb$, one can construct a 1-framing $\sigma+k$ by modifying $\sigma$ in a neighborhood of a point.
\end{proof}

\begin{remark}
In dimension $3$, the $1$-framings of Legendrian knots in the sense of \autoref{defn:1-framing} coincides with the usual notation of framings of knots, and the canonical $1$-framing corresponds to the so-called \emph{contact framing}.
\end{remark}

Taking a dual point of view, up to homotopy, any $1$-framing of $\Lambda$ determines locally a hypersurface $\Sigma \supset \Lambda$, and \textit{vice versa}. Using this terminology, we can rephrase \autoref{ex:collar nbhd} as follows: any Legendrian is contained in a hypersurface corresponding to the canonical $1$-framing as a regular Legendrian. As we will explain now, the same holds for any choice of $1$-framing.

\begin{prop} \label{prop:regular Leg}
Given a closed Legendrian $\Lambda \subset (M,\xi)$, any hypersurface $\Sigma$ containing $\Lambda$ can be $C^0$-small perturbed, relative to $\Lambda$, to a new hypersurface $\Sigma'$ such that $\Sigma'_{\xi}$ is Morse, and $\Lambda$ is regular with respect to $\Sigma'_{\xi}$. Moreover, $\Sigma'$ can be made Morse$^+$ if $\dim M \geq 5$.
\end{prop}

\begin{proof}
If $\dim M=3$, then given any $\sigma \in \Zbb$, one can construct by hand an annulus $\Sigma \supset \Lambda$ corresponding to the framing $\sigma$, with respect to which $\Lambda$ is regular. However, $\Sigma$ can be made Morse$^+$ only when $\tb_{\sigma}(\Lambda) \leq 0$. Assume $\dim M=2n+1 \geq 5$ for the rest of the proof. 
	
By the local nature of the Proposition, we can assume w.l.o.g. that as a smooth manifold $M = J^1(\Lambda)$, $\Lambda \subset M$ is the $0$-section, and $D^n \to \Sigma \to \Lambda$ is a (not necessarily trivial) disk bundle. The idea is to construct a contact structure $\xi_{\Lambda}$ on $M$ such that $\Lambda$ is Legendrian with respect to $\xi_{\Lambda}$, and $\Sigma$ satisfies all the properties of the Proposition. Then we argue that $\xi_{\Lambda}$ is isomorphic to $\xi$ by the Legendrian neighborhood theorem. For clarity, the proof is divided into three steps.
	
\s\n
\textsc{Step 1.} \emph{Construct a Morse vector field on $M$.}
	
\s
	
Let $v$ be a Morse vector field on $\Lambda$ with critical points $\xbf \coloneqq \{x_1,\dots,x_m\}$. Define a partial order on the set $\xbf$ by requiring $x_i \prec x_j$ if and only if there exists a flow line of $v$ from $x_i$ to $x_j$. Assume w.l.o.g. that if $x_i \prec x_j$, then $i<j$. Let $k_i$ be the Morse index of $x_i$. Clearly $k_1=0$ and $k_m=n$, but $k_i$ is not necessarily smaller than $k_j$ when $i<j$. Note that $v$ naturally induces a handle decomposition of $\Lambda$ such that each $x_i$ corresponds to a $k_i$-handle.
	
Extend $v$ to a Morse vector field $\bar{v}$ on $\Sigma$ such that the critical points of $v$ and $\bar{v}$ coincide. Such extension is by no means unique. For purely notational purposes, let us write $\bar{\xbf} \coloneqq \{ \bar{x}_1,\dots,\bar{x}_m \}$ for the set of critical points of $\bar{v}$ such that $\bar{x}_i=x_i$ for all $1 \leq i \leq m$. Let $\bar{k}_i$ be the Morse index of $\bar{x}_i$. Then clearly $k_i \leq \bar{k}_i$ for all $i$. We assign signs to each element in $\bar{\xbf}$ such that $\bar{x}_i$ is positive if $\bar{k}_i \leq n$ and $\bar{x}_i$ is negative if $\bar{k}_i \geq n$. Such sign assignment may not be unique, i.e., critical points of index $n$ can be either positive or negative.
	
\s\n
\textsc{Step 2.} \emph{Construct the contact form $\alpha_{\Lambda}$.}
	
\s
	
So far the discussion is purely topological. Now we describe how to construct the desired contact structure on $M$ using $\bar{v}$. Around any point $x \in \Sigma$, we can choose local coordinates $(z,\pbf,\qbf) \in \open(x) \subset M$, where $\pbf = (p_1,\dots,p_n), \qbf=(q_1,\dots,q_n)$, such that $\Sigma \cap \open(x) = \{z=0\}$ and $\Lambda \cap \open(x) = \{ z=\pbf=0 \}$. Throughout this paper, $\open$ denotes an unspecified small open neighborhood.
	
We construct a contact form $\alpha_{\Lambda}$ on $M$ as follows. First, we construct $\alpha_{\Lambda}$ near each $\bar{x}_i, 1 \leq i \leq m$. If $\bar{x}_i$ is positive, then $\bar{k}_i \leq n$ and we define
\begin{equation*}
\alpha_{\Lambda}|_{\open(\bar{x}_i)} = dz - q_1dp_1 - \dots - q_{\bar{k}_i}dp_{\bar{k}_i} + \dots + q_ndp_n - 2\pbf \cdot d\qbf,
\end{equation*}
where $\pbf \cdot d\qbf \coloneqq p_1dq_1 + \dots + p_ndq_n$. If $\bar{x}_i$ is negative, then $\bar{k}_i \geq n$ and $k_i \geq \bar{k}'_i$ where $\bar{k}'_i \coloneqq \bar{k}_i-n$. Define
\begin{equation*}
\alpha_{\Lambda}|_{\open(\bar{x}_i)} = -dz - 2\qbf \cdot d\pbf - p_1dq_1 - \dots - p_{\bar{k}'_i}dq_{\bar{k}'_i} + \dots + p_ndq_n.
\end{equation*}
Moreover, in both cases, the coordinates are chosen so that $\Lambda \cap \open(\bar{x}_i)$ is contained in $\{ z=q_{k_i+1}=\dots=q_n=p_1=\dots=p_{k_i}=0 \}$.
	
Following the arguments in \cite[Proposition 2.2.3]{HH19} and \cite[\S 9.2]{HH19}, one can extend $\alpha_{\Lambda}$ to $M$ such that $\Lambda$ is Legendrian with respect to $\xi_{\Lambda} \coloneqq \ker\alpha_{\Lambda}$, and the characteristic foliation $\Sigma_{\xi_{\Lambda}} = \bar{v}$. In other words, $\Lambda \subset \Sigma$ is a regular Legendrian with respect to $\xi_{\Lambda}$. By Legendrian neighborhood theorem, shrinking the neighborhood size of $\Lambda$ if necessary, we can assume $\xi$ is contactomorphic to $\xi_{\Lambda}$.
	
\s\n
\textsc{Step 3.} \emph{Compute the $1$-framing.}
	
\s
	
To complete the proof of the Proposition, it remains to compute the 1-framing $\sigma_{\Sigma}(\Lambda)$ of $\Lambda$ induced by $\Sigma$, which, in turn, is determined by $v$ together with the sign assignment. Using the same trick as in the proof of \autoref{lem:set of 1-framings}, we compute that 
\begin{equation} \label{eqn:compute 1-framing}
\sigma_{\Sigma}(\Lambda) = \sum_{x_i \text{ negative}} (-1)^{k_i}
\end{equation}
where sum is taken over all the negative singularities $x_i \in \Lambda$. In particular $\sigma_{\Sigma}(\Lambda)$ is independent of the extension $\bar{v}$.	
	
Assume $\bar{v}$ is Morse$^+$. Then there exists a (possibly disconnected) codimension-$1$ submanifold $\Gamma_{\Lambda} \subset \Lambda$, called the \emph{Legendrian divide}, which satisfies the following conditions:
\begin{itemize}
	\item[(LD1)] $\Gamma_{\Lambda}$ is everywhere transverse to $v$;
	\item[(LD2)] There exists the decomposition $\Lambda \setminus \Gamma_{\Lambda} = R_+(\Lambda) \cup R_-(\Lambda)$ such that the positive (resp. negative) critical points of $v$ are contained in $R_+(\Lambda)$ (resp. $R_-(\Lambda)$).
	\item[(LD3)] Near each component of $\Gamma_{\Lambda}$, $v$ flows from $R_+(\Lambda)$ to $R_-(\Lambda)$.
\end{itemize}
In this case, the $1$-framing induced by $\Sigma$ can be computed by
\begin{equation*}
\sigma_{\Sigma}(\Lambda) = (-1)^n \chi(R_-(\Lambda))
\end{equation*}
In particular, we recover \autoref{eqn:tb computation} by setting $n=1$ and observing that $R_-(\Lambda)$ is a disjoint union of intervals.
	
Since $\dim \Lambda \geq 2$ by assumption, one can choose $v,\bar{v}$ and sign assignment such that $\Sigma$ is Morse$^+$ and $\chi(R_-(\Lambda))$ takes any prescribed integer value.
\end{proof}

The following is an easy computation.

\begin{lemma} \label{lem:Leg self intersection number}
	Given a regular Legendrian $\Lambda$ in a Morse$^+$ hypersurface $\Sigma$, the self-intersection number of $\Lambda$ in $\Sigma$ is $\chi(R_+(\Lambda)) - \chi(R_-(\Lambda))$.
\end{lemma}

In particular, we say a regular Legendrian $\Lambda$ is \emph{balanced} if $\chi(R_+(\Lambda)) = \chi(R_-(\Lambda))$.

%\section{Stabilized equivalence of Legendrian knots in dimension 3}
%In this section we reprove the following theorem of Fuchs-Tabachnikov \cite{FT97} about (stabilized) Legendrian knots using convex surface theory, and view it as a toy model for our later study of higher-dimensional Legendrians.

%\begin{theorem}[Fuchs-Tabachnikov]
%Let $\Lambda_i \subset (M^3,\xi), i=0,1$, be Legendrian knots. If $\Lambda_0$ is smoothly isotopic to $\Lambda_1$, then $\Lambda_0^{\stab}$ is Legendrian isotopic to $\Lambda_1^{\stab}$, where $\Lambda_i^{\stab}$ is obtained from $\Lambda_i$ by a finite number of $(\pm)$-stabilizations.
%\end{theorem}

\section{Regular coLegendrians in dimension $5$} \label{sec:regular coLeg}
In this section, we study the Morse-theoretic structures of regular coLegendrians $Y \subset (M,\xi)$ introduced in \autoref{defn:coLeg submfd}. If $Y$ is a smooth submanifold, then it follows from \cite{Hua15} that $Y$ is naturally equipped with a (singular) Legendrian foliation $\Fcal \coloneqq \ker\alpha|_Y$, where $\alpha$ is a contact form. Conversely, the Legendrian foliation determines the germ of the contact structure near $Y$. However, smooth coLegendrians are not only difficult to find in general but also not sufficient for studying Legendrian isotopies. It turns out that the appropriate class of coLegendrians to study in this context contain certain ``cone-type'' singularities, which we will explain in details in this section.

\emph{For the rest of this section, we will assume $\dim M=5$ and $\dim Y=3$.} One major advantage in this dimension is the following obvious fact, which fails in higher dimensions.

\begin{lemma} \label{lem:coiso tangent to char fol}
A smooth $3$-submanifold $Y \subset (M,\xi)$ contained in a hypersurface $\Sigma$ is coLegendrian if it is tangent to $\Sigma_{\xi}$.
\end{lemma}

\begin{proof}
For any $x \in Y$, either $T_x Y \subset \xi_x$ or $T_x Y \cap \xi_x \subset \xi_x$ is $2$-dimensional. In the former case $Y$ is clearly coisotropic at $x$. In the latter case $T_x Y \cap \xi_x \subset \xi_x$ is a Lagrangian subspace since $\Sigma_{\xi}(x) \in T_x Y \cap \xi_x$ by assumption.
\end{proof}

\begin{remark}
Since we will discuss submanifolds $Y$ which are not everywhere smooth, we say $Y$ is tangent to a vector field $v$ if for any $x \in Y$, the flow line passing through $x$ is completely contained in $Y$. Under this convention, \autoref{lem:coiso tangent to char fol} can be generalized to non-smooth $Y$ and asserts that $Y$ is coLegendrian on the smooth part.
\end{remark}

The definition of \emph{regular coLegendrians} is completely parallel to the definition of regular Legendrians in \autoref{defn:regular legendrian}. Namely, with respect to a Morse hypersurface $\Sigma$ containing $Y$, we say $Y$ is \emph{regular} if it is tangent to $\Sigma_{\xi}$ and the restricted critical points of $\Sigma_{\xi}|_Y$ on $Y$ are nondegenerate. Note that the normal bundle of a regular coLegendrian is necessarily trivial since it is contained in a hypersurface by definition.

The section is organized as follows. In \autoref{subsec:coLeg handles}, we study models of coLegendrian handles which can be used to build any regular coLegendrian. In \autoref{subsec:coLeg existence}, we establish the existence of coLegendrians in the closed case. Then case of coLegendrians with Legendrian boundary is dealt with in \autoref{subsec:coLeg with bdry}.

\subsection{CoLegendrian handles} \label{subsec:coLeg handles}
Suppose $Y \subset \Sigma$ is a regular coLegendrian. It turns out that the (Morse) vector field $\Sigma_{\xi}|_Y$ itself is insufficient to determine the contact germ near $Y$. Indeed, it is the (singular) Legendrian foliation $\Fcal$ on $Y$, which determine the contact germ by \cite{Hua15}. The goal of this subsection is to work out local models of $\Fcal$ in the handles given by $\Sigma_{\xi}|_Y$.

\s\n
\textsc{Notation:}
\textit{Suppose $p \in \Sigma$ is a critical point of $\Sigma_{\xi}$. The Morse index of $p$ is called the \emph{$\Sigma$-index}. If, in addition, $p \in Y$, then the Morse index of $\Sigma_{\xi}|_Y$ at $p$ is called the \emph{$Y$-index}. This terminology extends to other regular submanifolds, e.g., Legendrians, in $\Sigma$ in the obvious way.}

\s
In the following, we will study \emph{coLegendrian handles}, i.e., the handles in $Y$ determined by $\Sigma_{\xi}|_Y$, and the associated Legendrian foliations $\Fcal$ in detail.

\s
\subsubsection{CoLegendrian handle $H_0$ of $Y$-index $0$} \label{subsubsec:H_0}

Let $p_0 \in H_0$ be the critical point. Then the $\Sigma$-index $\ind_{\Sigma} (p_0)=0$ or $1$. Denote the handle in $\Sigma$ corresponding to $p_0$ by $\widetilde{H}_0$, which is either a $0$-handle or a $1$-handle. Moreover, write $\p \widetilde{H}_0 = \p_+ \widetilde{H}_0 \cup \p_- \widetilde{H}_0$, where $\Sigma_{\xi}$ is inward-pointing along $\p_- \widetilde{H}_0$ and outward-pointing along $\p_+ \widetilde{H}_0$. Similarly, one can write $\p H_0 = \p_+ H_0 \cup \p_- H_0$ such that $\p_{\pm} H_0 \subset \p_{\pm} \widetilde{H}_0$, respectively, although $\p_- H_0 = \varnothing$ in this case. Note that $\p_{\pm} \widetilde{H}_0$ are naturally contact $3$-manifolds and $\p_+ H_0 \subset \p_+ \widetilde{H}_0$ is a $2$-sphere.

Observe that $p_0$ is necessarily positive. In what follows, we will always identify the characteristic foliation with the Liouville vector field for positive critical points, and the negative Liouville vector field for negative critical points.

\s\n
\textsc{Case 1.} $\ind_{\Sigma} (p_0)=1$.

\s
Identify $\widetilde{H}_0 \cong B^1 \times B^3$ such that  the Liouville vector field can be written as
\begin{equation} \label{eqn:index 1 vf}
X_1 \coloneqq -x_1\p_{x_1} + 2y_1\p_{y_1} + \tfrac{1}{2} (x_2\p_{x_2}+y_2\p_{y_2}),
\end{equation}
where $x_1 \in B^1$ and $(y_1,x_2,y_2) \in B^3$, and the Liouville form on $\widetilde{H}_0$, i.e., the restricted contact form, is
\begin{equation} \label{eqn:index 1 Liouville form}
\lambda_1 \coloneqq \alpha|_{\widetilde{H}_0} = -x_1 dy_1-2y_1 dx_1 + \tfrac{1}{2} (x_2 dy_2 - y_2 dx_2).
\end{equation}
Under this identification, we have $H_0 \cong \{0\} \times B^3$ is the unstable manifold of $p_0$, which is of course smooth. Moreover, the Legendrian foliation $\Fcal_{H_0}$ on  $H_0$ is defined by
\begin{equation} \label{eqn:Leg foliation on H0}
\Fcal_{H_0} = \ker(\lambda_1|_{H_0}) = \ker(x_2 dy_2 - y_2 dx_2).
\end{equation}
It follows that the characteristic foliation on $\p H_0$, which is nothing but $\Fcal_{H_0} \cap \p H_0$, is standard, i.e., there are one source and one sink and all flow lines travel from the source to the sink.

\begin{remark}
	The particular choice of the Liouville vector field in \autoref{eqn:index 1 vf} (and the Liouville form) is somewhat arbitrary. Two different choices of such Liouville forms differ by an exact $1$-form, and we say the different choices are \emph{deformation equivalent}. Note that deformation equivalence is strictly weaker than (symplectic) isotopy. This remark applies to all the particular choices of Liouville forms in subsequent models.
\end{remark}

\s\n
\textsc{Case 2.} $\ind_{\Sigma} (p_0)=0$.

\s
Identify $\widetilde{H}_0 \cong B^4$ such that the Liouville vector field can be written as
\begin{equation*}
X_0 \coloneqq \tfrac{1}{2} (x_1\p_{x_1} + y_1\p_{y_1} + x_2\p_{x_2} + y_2\p_{y_2}),
\end{equation*}
and the Liouville form
\begin{equation} \label{eqn:index 0 Liouville form}
\lambda_0 \coloneqq \alpha|_{\widetilde{H}_0} = \tfrac{1}{2} (x_1 dy_1 - y_1 dx_1 + x_2 dy_2 - y_2 dx_2).
\end{equation}

Observe that $(\p \widetilde{H}_0, \lambda_0|_{\p \widetilde{H}_0}) \cong (S^3,\xi_{\std})$ and $\p H_0=\p_+ H_0 \subset \p \widetilde{H}_0$ can be identified with a $2$-sphere in the standard contact $S^3$. In particular $\xi_{\std}$ induces a characteristic foliation $(\p H_0)_{\xi_{\std}}$ on $\p H_0$. It follows that $H_0$ is the cone over $\p H_0$ and the Legendrian foliation $\Fcal_{H_0}$ is also the cone over $(\p H_0)_{\xi_{\std}}$. Namely, a leaf of $\Fcal_{H_0}$ is the cone over a leaf of $(\p H_0)_{\xi_{\std}}$. \emph{Hereafter all cones are taken with respect to appropriate Liouville vector fields, which is $X_0$ in this case.} Note that, in this case, $H_0$ is smooth only when $\p H_0$ is equatorial.

\s
\subsubsection{Positive coLegendrian handle $H^+_1$ of $Y$-index $1$} \label{subsubsec:poitive H_1}

Let $p_1 \in H^+_1$ be the critical point. Then $\ind_{\Sigma} (p_1) = 1$ or $2$. We continue using the terminologies from \autoref{subsubsec:H_0} to denote the corresponding handle in $\Sigma$ by $\widetilde{H}^+_1$. 

\s\n
\textsc{Case 1.} $\ind_{\Sigma} (p_1)=2$.

\s
Identify $\widetilde{H}^+_1 \cong B^2 \times B^2$ such that the Liouville vector field can be written as
\begin{equation} \label{eqn:index 2 vf}
X_2 \coloneqq -x_1\p_{x_1}-x_2\p_{x_2}+2y_1\p_{y_1}+2y_2\p_{y_2},
\end{equation} 
and the Liouville form
\begin{equation} \label{eqn:index 2 Liouville form}
\lambda_2 \coloneqq \alpha|_{\widetilde{H}^+_1} = -x_1 dy_1 - x_2 dy_2 - 2y_1 dx_1 - 2y_2 dx_2,
\end{equation}
where $(x_1,x_2) \in B^2$ in the first component and $(y_1,y_2) \in B^2$ in the second component.

To see the embedding $H^+_1 \subset \widetilde{H}^+_1$, observe that the unstable disk in $H^+_1$ coincides with the unstable disk in $\widetilde{H}^+_1$ for index reasons. On the other hand, the $1$-dimensional stable disk in $H^+_1$ sits in the stable disk $B^2_x \coloneqq B^2 \times \{0\}$ in $\widetilde{H}^+_1$, and is tangent to the restricted Liouville vector field
\begin{equation*}
X_2|_{B^2_x} = -x_1\p_{x_1}-x_2\p_{x_2}.
\end{equation*}
In other words, $H^+_1 = \delta \times B^2 \subset \widetilde{H}^+_1$ where $\delta \subset B^2_x$ is the union of two (different) radii. In general $H^+_1$ has corners along $\{0\} \times B^2$, and is smooth precisely when $\delta$ is a diameter.

Next we turn to the Legendrian foliation $\Fcal_{H^+_1}$ on $H^+_1$. Suppose $\overline{\delta} \subset \delta$ is one of the two radii and suppose further w.l.o.g. that $\overline{\delta}=\{x_1 \geq 0, x_2=0\} \subset B^2_x$. It suffices to understand the Legendrian foliation $\Fcal_{H^+_1}|_{\overline{\delta} \times B^2}$ on one half of $H^+_1$, which is given by
\begin{equation*}
\Fcal_{H^+_1}|_{\overline{\delta} \times B^2} = \ker(\lambda_2|_{\overline{\delta} \times B^2}) = \ker(x_1 dy_1+2y_1 dx_1).
\end{equation*}
In particular, the characteristic foliation on each (disk) component of $\p_- H^+_1$ is a linear foliation, and on $\p_+ H^+_1$, which is an annulus with corners along $\{0\} \times \p B^2$, it is as shown in \autoref{fig:cornered char foliation}. In particular, observe that there are four half-saddle points on $\p_+ H^+_1$, all of which lie on $\{0\} \times \p B^2$. Let $p_1,p_2$ be the two half-saddles on $\p_+ (\overline{\delta} \times B^2)$, and $q_1,q_2$ be the other two half-saddles. Then the relative positions between $p_1$ and $p_2$, as well as between $q_1$ and $q_2$, are fixed. However, the relative position between $p_1$ and $q_1$ (or equivalently, $p_2$ and $q_2$) depends on the angle of $\delta$ at the origin. In particular, if the angle is $\pi$, i.e., $\delta$ is a diameter, then $p_1$ (resp. $p_2$) collides with $q_1$ (resp. $q_2$), and the characteristic foliation on the smooth $\p_+ H^+_1$ possesses two (full) saddles. Finally, if we fix an orientation of $\p_+ H^+_1$, then $p_1$ and $p_2$ (resp. $q_1$ and $q_2$) always have opposite signs, and the characteristic foliation is oriented in such a way that on $\{0\} \times B^2$, it flows from the positive half-saddle to the negative half-saddle.

\begin{figure}[ht]
	\begin{overpic}[scale=.3]{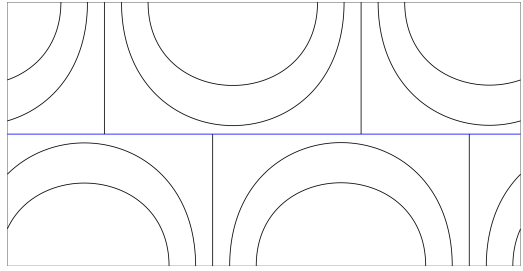}
	\end{overpic}
	\caption{The characteristic foliation on $\p_+ H^+_1$. The left and right sides are identified and the corners are along the blue circle.}
	\label{fig:cornered char foliation}
\end{figure}

\s\n
\textsc{Case 2.} $\ind_{\Sigma} (p_1)=1$.

\s
Identify $\widetilde{H}^+_1 \cong B^1 \times B^3$ such that the Liouville vector field and the Liouville form are given by \autoref{eqn:index 1 vf} and \autoref{eqn:index 1 Liouville form}, respectively. Continue using the notations from \autoref{subsubsec:H_0}, we write $B^3_0 \coloneqq \{0\} \times B^3$. Then the embedding $H^+_1 \subset \widetilde{H}^+_1$ takes the form $H^+_1 = B^1 \times K(\gamma)$, where $\gamma \subset \p B^3_0$ is a closed loop and $K(\gamma) \subset B^3_0$ is the cone over $\gamma$ taken with respect to the vector field
\begin{equation} \label{eqn:skewed radial vf}
X_1|_{B^3_0} = 2y_1 \p_{y_1} + \tfrac{1}{2} (x_2 \p_{x_2}+y_2 \p_{y_2}).
\end{equation}

Assume $\gamma$ is smooth and \emph{generic} in the following sense. Consider the foliation $\Gcal$ on $B^3_0$ defined by
\begin{equation} \label{eqn:char fol on S2}
\Gcal \coloneqq \ker(\lambda_1|_{B^3_0}) = \ker(x_2dy_2-y_2dx_2) = \ker (r^2d\theta),
\end{equation}
where $(r,\theta)$ denotes the polar coordinates on the $x_2y_2$-plane. It induces a foliation $\Gcal|_{\p B^3_0} \coloneqq \Gcal \cap \p B^3_0$ which is singular at the north pole $(1,0,0)$ and the south pole $(-1,0,0)$. We say $\gamma$ is \emph{generic} if the following hold:
\begin{itemize}
	\item[(Gen1)] $\gamma$ does not pass through the north and the south poles.
	\item[(Gen2)] The intersection between $\gamma$ and a leaf of $\Gcal|_{\p B^3_0}$ is either transversal or quadratic tangential, i.e., modeled on the intersection between the $u$-axis in $\Rbb^2_{u,v}$ and the graph of $v=u^2$. In particular, the tangential points are isolated.
	\item[(Gen3)] None of the quadratic tangential points lie on the equator $\p B^3_0 \cap \{y_1=0\}$.
\end{itemize}

To understand the Legendrian foliation $\Fcal_{H^+_1}$ on $H^+_1$, it turns out to be convenient to zoom in on a small neighborhood of the critical point $p_1 \in \widetilde{H}^+_1$. Motivated by this, let $R_1,R_2$ be the radii of $B^1,B^3$, respectively. More explicitly, $B^1=\{\abs{x_1} \leq R_1\}$ and $B^3=\{\abs{y_1}^2+r^2 \leq R_2^2\}$. Before getting into the details, let's briefly explain the strategy to visualize $\Fcal$ as follows. In all previous cases, we first describe the Legendrian foliation $\Fcal$, which is defined by a relatively simple $1$-form, on the relevant handle, say, $H$, and then examine its trace on the boundaries $\p_{\pm} H$. In this case, however, the above procedure will be reversed due to the more complicated structure of $\Fcal_{H^+_1}$. Namely, we will first describe the trace of $\Fcal_{H^+_1}$ on $\p_{\pm} H^+_1$, i.e., the characteristic foliations, and then use it to describe $\Fcal_{H^+_1}$. 

Let $w_1,\dots,w_m \in \gamma$ be the quadratic tangential points introduced in (Gen2). We first analyze the characteristic foliation on $\p_+ H^+_1 = B^1 \times \gamma$, which is the easier part. Fix a orientation of $\gamma$ and let $\dot{\gamma}$ be the positive tangent vector. Identify $\gamma$ with $\{0\} \times \gamma$. If we denote the restriction of the vector field $\p_{x_1}$ on $\p_+ H^+_1$ along $\gamma$ by $\p_{x_1}|_{\gamma}$, then observe that $\lambda_1 (\p_{x_1}|_{\gamma}) = -2y_1(\gamma)$, where $y_1(\gamma)$ denotes the $y_1$-coordinate of the points on $\gamma$. Together with (Gen3), we see that the characteristic foliation on $\p_+ H^+_1$ is nonsingular for $R_1$ sufficiently small and is as shown in \autoref{fig:char fol positive bdry}, where the tangencies between $\gamma$ and the characteristic foliation are in one-to-one correspondence with the $w_i$'s.

\begin{figure}[ht]
	\begin{overpic}[scale=.3]{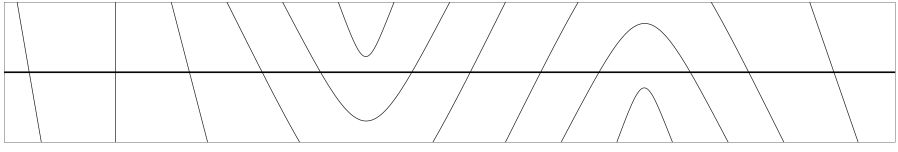}
	\end{overpic}
	\caption{The characteristic foliation on $\p_+ H^+_1$. The left and right sides are identified.}
	\label{fig:char fol positive bdry}
\end{figure}

Next, we turn to the characteristic foliation on $\p_- H^+_1$, which consists of two disks. In what follows we consider the component of $\p_- H^+_1$ with $x_1=R_1>0$. The other component with $x_1=-R_1$ can be analyzed similarly. Cut $\gamma$ open at the $w_i$'s to obtain $m$ consecutive open segments $\gamma_1,\dots,\gamma_m$ such that $\gamma_i$ denotes the segment between $w_i$ and $w_{i+1}$, where $1 \leq i \leq m$ and $m+1$ is identified with $1$. Let $K(\gamma_i)$ be the cone over $\gamma_i$. For definiteness, let's consider $K(\gamma_1)$ and suppose for simplicity that the span of $\theta(\gamma_1)$ is less than $2\pi$, where $\theta(\gamma_1)$ denotes the $\theta$-coordinate of the points on $\gamma_1$. The general case can be dealt with similarly. Let $\overline{\gamma}_1$ be the projection of $\gamma_1$ to the $x_2y_2$-plane, and $K(\gamma_1)$ be the cone over $\gamma_1$, taken with respect to the radial vector field $r\p_r$. Then $K(\overline{\gamma}_1) \subset \Rbb^2_{x_2,y_2}$ is an embedded sector, over which $K(\gamma_1)$ is graphical and can be written as
\begin{equation} \label{eqn:cone gamma_1}
K(\gamma_1) = \{ y_1=f(\theta)r^4 ~|~ (r,\theta) \in K(\overline{\gamma}_1) \},\footnote{The term $r^4$ comes from the particular choice of the Liouville vector field in \autoref{eqn:index 1 vf}.}
\end{equation}
such that $f'(\theta)$ blows up as $\theta$ approaches $\theta_{\min}$ or $\theta_{\max}$, where $\theta_{\min},\theta_{\max}$ are the lower and upper limits of the $\theta$-coordinate in $K(\overline{\gamma}_1)$.

We are interested in the characteristic foliation $K(\gamma_1)_{\xi}$ on $K(\gamma_1)$. Let $\overline{K(\gamma_1)}_{\xi}$ be the projection of $K(\gamma_1)_{\xi}$ to $K(\overline{\gamma}_1)$. We have
\begin{align*}
\overline{K(\gamma_1)}_{\xi} &= \ker (-R_1 d(f r^4) + \tfrac{1}{2} r^2 d\theta) \\
&= \ker ((\tfrac{1}{2} r^2 - R_1 f' r^4) d\theta - 4R_1 f r^3 dr).
\end{align*}
We claim that $\overline{K(\gamma_1)}_{\xi}$ is nonzero away from the origin if $R_2$ is sufficiently small. Indeed, away from the origin, the $dr$ component vanishes precisely when $f$ vanishes. But at these points $f'$ is finite due to (Gen3), and therefore the $d\theta$ component is nonzero for $r$ sufficiently small. 

The key to visualize $\overline{K(\gamma_1)}_{\xi}$ consists of two observations. First, note that the $dr$ component is nonvanishing whenever $f$ is nonvanishing. Second, the $d\theta$ component can possibly vanish only near $\theta_{\max}$ and $\theta_{\min}$, where $f'$ blows up. Let's consider $\theta_{\max}$ here and leave the discussion of $\theta_{\min}$ to the interested reader. If $\lim_{\theta \to \theta_{\max}} f'(\theta)=-\infty$, then the $d\theta$ component is never zero near $\theta_{\max}$. Hence we can assume that $\lim_{\theta \to \theta_{\max}} f'(\theta)=+\infty$. In this case, for each $\theta$ sufficiently close to $\theta_{\max}$, there exists a unique point $(r(\theta),\theta) \in \overline{K(\gamma_1)} \setminus \{0\}$ at which $\overline{K(\gamma_1)}_{\xi}=\ker(dr)$. Moreover, the sequence of points $(r(\theta),\theta)$ converge to the origin as $\theta \to \theta_{\max}$. 

At this point, we note that the dynamics of the characteristic foliation is better understood on the entire $K(\gamma)$ rather than on each individual $K(\gamma_i)$. More precisely, let $\nu(w_i) \subset \gamma$ be a neighborhood of $w_i$, and consider the cone $K(\nu(w_i)) \subset K(\gamma)$. Observe that if $f_{i-1}$ and $f_i$ are the defining angular functions for $\gamma_{i-1}$ and $\gamma_i$, respectively, as in \autoref{eqn:cone gamma_1}, then either $\lim_{\theta \to \theta_i} f_{i-1}'(\theta)=-\infty$ or $\lim_{\theta \to \theta_i} f_i'(\theta)=-\infty$, where $\theta_i$ denotes the angular coordinate of $w_i$. Hence by the observations made above, the restriction of $K(\gamma_i)_{\xi}$ to $K(\nu(w_i))$ is either one of the two scenarios shown in \autoref{fig:char foliation negative bdry}. Finally, note that away from the $K(\nu(w_i))$'s, the flows line of $K(\gamma)_{\xi}$ simply go from the origin towards $\gamma$.

\begin{figure}[ht]
	\begin{overpic}[scale=.4]{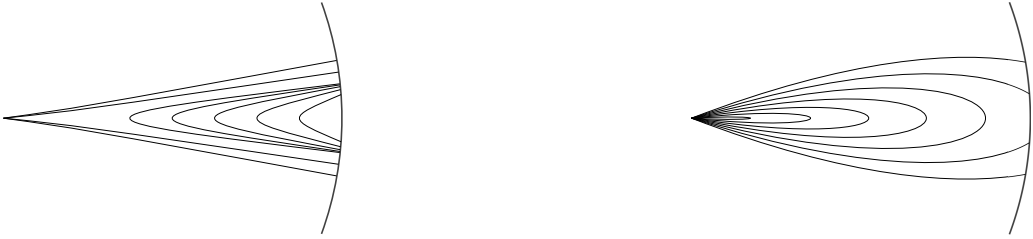}
	\end{overpic}
	\caption{Two possibilities of the restriction of $K(\gamma_i)_{\xi}$ to $K(\nu(w_i))$.}
	\label{fig:char foliation negative bdry}
\end{figure}

Finally, let's describe the Legendrian foliation $\Fcal_{H^+_1}$, which, in fact, can be read off from the characteristic foliation $(\p_+ H^+_1)_{\xi}$ on $\p_+ H^+_1$ (cf. \autoref{fig:char fol positive bdry}) as follows. Recall the vector field
\begin{equation*}
X_1|_{H^+_1} = -x_1 \p_{x_1} + \tfrac{1}{2} (x_2 \p_{x_2}+y_2 \p_{y_2}),
\end{equation*}
which is transverse to $\p_+ H^+_1$. Then each leaf $F$ of $\Fcal_{H^+_1}$ can be visualized as the totality of trajectories of $X_1|_{H^+_1}$ which pass through a leaf of $(\p_+ H^+_1)_{\xi}$. Recall that a leaf $\ell$ of $(\p_+ H^+_1)_{\xi}$ is always a properly embedded arc. Let $F(\ell)$ be the leaf of $\Fcal_{H^+_1}$ such that $F(\ell) \cap \p_+ H^+_1=\ell$. Then we have the following possibilities for the shape of $F(\ell)$ depending on the position of $\ell \subset \p_+ H^+_1$:
\begin{itemize}
	\item Suppose $\p\ell$ is contained in one component of $\p(\p H^+_1)$. Then
	\begin{itemize}
		\item if $\ell \cap \gamma=\varnothing$, then $F(\ell)$ is a disk as shown in \autoref{fig:Leg foliation on 1 handle}(a);
		\item if $\ell$ is tangent to $\gamma$, then $F(\ell)$ is a disk as shown in \autoref{fig:Leg foliation on 1 handle}(b);
		\item if $\ell$ intersects $\gamma$ is two points, then $F(\ell)$ is an annulus as shown in \autoref{fig:Leg foliation on 1 handle}(c).
	\end{itemize}
	\item Suppose the two points $\p\ell$ are contained in different components of $\p(\p H^+_1)$. Then $F(\ell)$ is a strip as shown in \autoref{fig:Leg foliation on 1 handle}(d).
\end{itemize}

\begin{figure}[ht]
	\begin{overpic}[scale=.35]{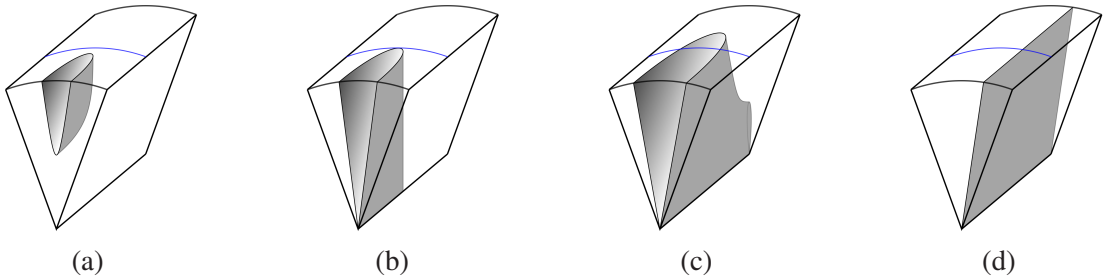}
		\put(6,-3.5){(a)}
		\put(33.7,-3.5){(b)}
		\put(61.5,-3.5){(c)}
		\put(89,-3.5){(d)}
	\end{overpic}
	\vspace{4mm}
	\caption{Different leaves of $\Fcal_{H^+_1}$. The blue arc represents $\gamma \subset \p_+ H^+_1$.}
	\label{fig:Leg foliation on 1 handle}
\end{figure}

\s\n
\textsc{Comparison between Case 1 and Case 2.}

\s
Since our interests lie in $H^+_1$ and its Legendrian foliation $\Fcal_{H^+_1}$ and not in the ambient handle $\widetilde{H}^+_1$, which may be either a $1$-handle (Case 2) or a $2$-handle (Case 1), it is instructive to compare the two models and understand their differences. To avoid confusions, conflicting notations used in (Case 1) and (Case 2) will be decorated by $(1)$ and $(2)$, respectively.

First of all, observe that the attaching region $\p_- H^{+,(1)}_1$ is always smooth, but this is not the case for $\p_- H^{+,(2)}_1$. Indeed, in the generic case, i.e., (Gen1)--(Gen3) are satisfied, a necessary condition for $\p_- H^{+,(2)}_1$ to be smooth is the nonexistence of (quadratic) tangencies on $\gamma$ (cf. (Gen2)). Equivalently, it means that $\gamma$ can be written in the form of \autoref{eqn:cone gamma_1} for a globally defined $f(\theta), \theta \in [0,2\pi]$. Strictly speaking, one also needs to assume that the oscillation of $f(\theta)$ is not too rapid to guarantee the smoothness of $K(\gamma)$. However, this technical point will not be important for us. The key observation here is that all the leaves of the Legendrian foliation $\Fcal_{H^+_1}^{(2)}$ are of the form shown in \autoref{fig:Leg foliation on 1 handle}(d), at least when $R_1,R_2$ are small.

Next, let's consider the situation where $\p_- H^{+,(1)}_1$ is smooth, i.e., $\delta$ is a diameter. In this case, the Legendrian foliation $\Fcal_{H^+_1}^{(1)}$ coincides with $\Fcal_{H^+_1}^{(2)}$ when $\mu \coloneqq \gamma$ is a meridian great circle. Note, however, that $\mu$ is not generic since it fails (Gen1)--(Gen3). In what follows we consider a particularly simple perturbation of $\mu$ so that it becomes generic and try to understand how $\Fcal_{H^+_1}^{(2)}$ changes. Continue using notations from (Case 2), suppose w.l.o.g. that $\mu=\p B^3_0 \cap \{y_2=0\}$. Let $\tau: \p B^3_0 \to \p B^3_0$ to a small rotation about the $x_2$-axis. Then $\tau(\mu)$ is generic. Indeed, it is everywhere transverse to $\Gcal|_{\p B^3_0}$. If we fix $R_1=R_2=1$, instead of letting them shrink as in (Case 2), then the characteristic foliation on one component of $\p_- H^{+,(2)}_1$ is Morse and has precisely two critical points: one source at the center and one saddle, which are in the canceling position. Moreover, as the angle of rotation $\tau$ tends to zero, the saddle approaches towards the source at the center, and cancels it in the limit\footnote{This procedure is nothing but a realization of Giroux's elimination lemma.}. On the other hand, if we fix the angle of rotation $\tau$ and shrink $R_1,R_2$, then we recover the smooth $H^{+,(2)}_1$ discussed in the previous paragraph.

\s
\subsubsection{Negative coLegendrian handle $H^-_1$ of $Y$-index $1$} \label{subsubsec:negative H_1}
Let $p_1 \in H^-_1$ be the critical point. Then $\ind_{\Sigma} (p_1)=2$ necessarily, and we are in the same situation as in \autoref{subsubsec:poitive H_1} (Case 1). Note that for negative critical points, the characteristic foliation (viewed as a vector field) and the Liouville vector field differ by a sign.

There is only one difference between the positive and the negative case which we now explain. Recall from \autoref{subsubsec:poitive H_1} (Case 1) that the characteristic foliation on $\p_+ H^+_1$ has four half-saddles, which come in two pairs $p_1,p_2$ and $q_1,q_2$. Moreover, with respect to a given orientation of $\p_+ H^+_1$, $p_1$ and $p_2$ have opposite signs and the stable manifold of the positive one coincides with the unstable manifold of the negative one. The same applies to $q_1$ and $q_2$. Now in the negative case, $\p_+ H^-_1$ also has four half saddles, which we denote by $p^-_1,p^-_2$ and $q^-_1,q^-_2$. As before, $p_1^-$ and $p_2^-$ have opposite signs, but in this case, the unstable manifold of the positive one coincides with the stable manifold of the negative one. In other words, there exist two flow lines coming from the negative (half) saddle and flowing into the positive one. The same applies to $q_1^-,q_2^-$.

\s
\subsubsection{Positive coLegendrian handle $H^+_2$ of $Y$-index $2$} \label{subsubsec:positive H_2}
Let $p_2 \in H^+_2$ be the critical point. Then $\ind_{\Sigma} (p_2)=2$ necessarily. Identify the ambient handle $\widetilde{H}^+_2 \cong B^2 \times B^2$ and continue using the model from \autoref{subsubsec:poitive H_1} (Case 1). In particular, the Liouville vector field and the Liouville form are given by \autoref{eqn:index 2 vf} and \autoref{eqn:index 2 Liouville form}, respectively. This case is completely dual to \autoref{subsubsec:negative H_1} in the sense that the characteristic foliations on $\p_{\pm} H_2^+$ can be identified with those on $\p_{\mp} H^-_1$, respectively, and the Legendrian foliations $\Fcal_{H^+_2}$ and $\Fcal_{H^-_1}$ coincide with a flip of coordinates.

\s
\subsubsection{Negative coLegendrian handle $H^-_2$ of $Y$-index $2$} \label{subsubsec:negative H_2}
Let $p_2 \in H^-_2$ be the critical point. Then $\ind_{\Sigma} (p_2)=2$ or $3$. Now the $\ind_{\Sigma} (p_2)=2$ case is dual to \autoref{subsubsec:poitive H_1} (Case 1) and the $\ind_{\Sigma} (p_2)=3$ case is dual to \autoref{subsubsec:poitive H_1} (Case 2). We omit the details. 

\s
\subsubsection{CoLegendrian handle $H_3$ of $Y$-index $3$} \label{subsubsec:H_3}
Let $p_3 \in H_3$ be the critical point. Then $\ind_{\Sigma} (p_3)=3$ or $4$. In either case $H_3$ is necessarily a negative handle. Here the $\ind_{\Sigma} (p_3)=3$ case is dual to \autoref{subsubsec:H_0} (Case 1) and the $\ind_{\Sigma}(p_3)=4$ case is dual to \autoref{subsubsec:H_0} (Case 2). We omit the details.

\begin{remark}
	As a concluding remark to our constructions of coLegendrian handles, note that these handles as constructed  not necessarily smooth and there may be cones, corners and families of cones. However, from a purely topological point of view, the smoothness regularity of the handles can be much lower than those considered in this subsection. For example, the loop $\gamma$ considered in \autoref{subsubsec:negative H_1} (Case 2) are assumed to be smooth for no obvious reasons. Our choices will be justified in the next subsection where we study the existence of coLegendrians.
\end{remark}

\subsection{Existence of coLegendrians} \label{subsec:coLeg existence}
The goal of this subsection is to prove the following result on coLegendrian approximation.

\begin{prop} \label{prop:coLeg approx}
	Suppose $Y \subset (M^5,\xi)$ is a closed $3$-submanifold with trivial normal bundle. Then $Y$ can be $C^0$-approximated by a regular coLegendrian with isolated cone singularities.
\end{prop}

\begin{proof}
	The proof essentially consists of two steps\footnote{The weight of the two steps may seem extremely imbalanced: Step 2 is some ten times longer than Step 1. But the truth is that Step 1 relies on all of \cite{HH19}, which is some ten times longer than Step 2.}. The first step is to approximate $Y$ by a regular coLegendrian with various singularities appeared in \autoref{subsec:coLeg handles}, and the second step is to eliminate all singularities except an isolated collection of cones.
	
	\s\n
	\textsc{Step 1.} \textit{Topological approximation.}
	
	\s
	Consider a hypersurface $\Sigma \coloneqq Y \times [-1,1] \subset M$ such that $Y$ is identified with $Y \times \{0\}$. By the existence $h$-principle for contact submanifolds in \cite{HH19}, we can assume, up to a $C^0$-small perturbation of $\Sigma$, that $\Sigma_{\xi}=\p_s$ where $s$ denotes the coordinate on $[-1,1]$. Again by the folding techniques developed in \cite{HH19}, one can further $C^0$-perturb $\Sigma$ such that with respect to the new Morse vector field $\Sigma_{\xi}$, there exists a (topological) copy of $Y$ satisfying the following 
	\begin{itemize}
		\item[(RA1)] $Y$ is tangent to $\Sigma_{\xi}$;
		\item[(RA2)] $\Sigma_{\xi}|_Y$ is Morse;
		\item[(RA3)] $\Sigma_{\xi}$ is inward pointing along the $1$-dimensional transverse direction to $Y$.
	\end{itemize}
	Note that there also exists a (disjoint) copy of $Y$ which satisfies all the above conditions but replacing ``inward pointing'' by ``outward pointing'' in (RA3). Our choice here is completely arbitrary. 
	
	In this way, we have constructed a $C^0$-approximation of $Y$ which is regular and coisotropic according to \autoref{lem:coiso tangent to char fol}. \emph{By abusing notations, we will denote the approximating regular coisotropic submanifold by $Y$ in what follows.} However, such $Y$ may not be everywhere smooth and our next task is to analyze its singularities. 
	
	\s\n
	\textsc{Step 2.} \textit{Smoothing of singularities.}
	
	\s
	Observe that for any critical point $p \in Y$, we have $\ind_Y (p)+1=\ind_{\Sigma} (p)$ by (RA3). This is a rather strong constraint on the structure of $Y$. We will make use of this rigidity in the beginning of the argument and gradually get rid of it as more flexibility becomes necessary. For clarity, this step is further subdivided into substeps according to the $Y$-index of the handles.
	
	\s\n
	\textsc{Substep 2.1.} \textit{The $0$-handles $H_0$.}
	
	\s
	According to \autoref{subsubsec:H_0} (Case 1), $H_0$ is a smooth $3$-ball with boundary $\p H_0$. Hence there is nothing to smooth. Note that the characteristic foliation on $\p H_0$ is standard (cf. \autoref{eqn:Leg foliation on H0}).
	
	\s\n
	\textsc{Substep 2.2.} \textit{Round the $1$-handles $H_1^{\pm}$.}
	
	\s
	We only discuss the case of $H^+_1$ and note that the case of $H^-_1$ is similar. According to \autoref{subsubsec:poitive H_1} (Case 1), the ambient handle $\widetilde{H}^+_1 \cong B^2 \times B^2$ comes with the Liouville form $\lambda_2$ given by \autoref{eqn:index 2 Liouville form}, and $H_1^+ = \delta \times B^2$ where $\delta \subset B^2_x$ is the union of two radii $\delta_1$ and $\delta_2$. It follows that $H_1^+$ is smooth exactly when $\delta$ is a diameter. The goal of this substep is apply a Hamiltonian perturbation to $\widetilde{H}^+_1$ to, in effect, round $\delta$ and hence also $H^+_1$. Roughly speaking, the idea is that the Hamiltonian isotopy, when restricted to $B^2_x$, generates a partial rotation which rotates, say, $\delta_2$ to an angle opposite to that of $\delta_1$. See \autoref{fig:round radii}.
	
	\begin{figure}[ht]
		\begin{overpic}[scale=.4]{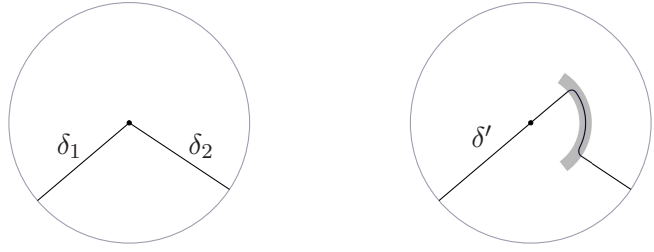}
			\put(7.5,14){$\delta_1$}
			\put(28,14){$\delta_2$}
			\put(72,15){$\delta'$}
		\end{overpic}
		\caption{Smoothing of $\delta=\delta_1 \cup \delta_2$.}
		\label{fig:round radii}
	\end{figure}
	
	To carry out the details, let's introduce polar coordinates $(r,\theta) \in B^2_x$ and dual coordinates $(r^{\ast},\theta^{\ast}) \in B^2_y$ defined by
	\begin{equation*}
	r^{\ast} \coloneqq y_1\cos\theta+y_2\sin\theta \quad\text{ and }\quad \theta^{\ast} \coloneqq y_2r\cos\theta-y_1r\sin\theta.
	\end{equation*}
	One can check that under this change of coordinates
	\begin{equation*}
	\omega=dx_1 \wedge dy_1+dx_2 \wedge dy_2 = dr \wedge dr^{\ast} + d\theta \wedge d\theta^{\ast}.
	\end{equation*}
	Consider a Hamiltonian function $H=\rho_1(r) \rho_2(r^{\ast}) \rho_3(\theta) \rho_4(\theta^{\ast})$ on $\widetilde{H}^+_1$ such that
	\begin{itemize}
		\item $\rho_1,\rho_2,\rho_3$ are $C^1$-small bump functions such that $\rho_1$ is supported in a near $r=r_0 \neq 0$, $\rho_2$ is supported near $r^{\ast}=0$, and $\rho_3$ is supported in an interval in $S^1$ of length $<\pi$ which intersects both $\delta_2$ and $-\delta_1$, but not $\delta_1$. See the shaded region in the right-hand-side of \autoref{fig:round radii}.
		\item $\rho_4$ is supported near $\theta^{\ast}=0$ and $\rho_4(0)=0$ but $\rho'_4(0) \neq 0$.
	\end{itemize}
	Observe that the Hamiltonian isotopy $\phi_H$ induced by $H$ leaves $B^2_x$ invariant since $H|_{B^2_x} \equiv 0$. Indeed, $\phi_H|_{B^2_x}$ is a partial rotation supported in $\supp(\rho_1(r) \rho_3(\theta))$ (e.g., the shaded region in the right-hand-side of \autoref{fig:round radii}), and the angle of rotation depends on $r,\theta$ and $\rho'_4(0)$. 
	
	Now consider the deformed Liouville form $\lambda'_2 \coloneqq \lambda_2+dH$ and the associated Liouville vector field $X'_2$. Then for appropriate choices of $\rho_1,\dots,\rho_4$, one can find a diameter $\delta' \subset B^2_x$ with respect to $X'_2$, i.e., a properly embedded smooth arc which is tangent to $X'_2$ and passes through the origin, such that $\delta'$ agrees with $\delta$ near $\p B^2_x$. In fact, by choosing $r_0$ sufficiently small, we can arrange so that $\delta$ agrees with $\delta'$ outside of a small neighborhood of the origin. Hence we have constructed a smoothed handle $H_1^{+,\op{sm}} \coloneqq \delta' \times B^2$ in the deformed $\widetilde{H}_1^+$.
	
	By construction, the smoothed $1$-handles $H_1^{+,\op{sm}}$ are attached to the $0$-handles in the same way that the original $1$-handles $H_1^+$ are attached. It remains to argue that the above smoothing operation does not affect the subsequent $2$ and $3$-handle attachments. Indeed, let $Y^{(1)}$ be the union of $0$ and $1$-handles before smoothing. Then a $2$-handle $H^-_2$ (which is a slice of a $3$-handle in $\Sigma$) is attached along a loop $\gamma \subset \p Y^{(1)}$. Note that generically, $\gamma$ may not be smooth since $\p Y^{(1)}$ is not smooth in general. In fact, $\gamma$ may have corners exactly where $\p Y^{(1)}$ has corners. However, the smoothing of the $1$-handles as described above simultaneously smooths the $\gamma$'s. See \autoref{fig:round corner cont'd}. Finally, the $3$-handle attachments are clearly not affected.
	
	\begin{figure}[ht]
		\begin{overpic}[scale=.5]{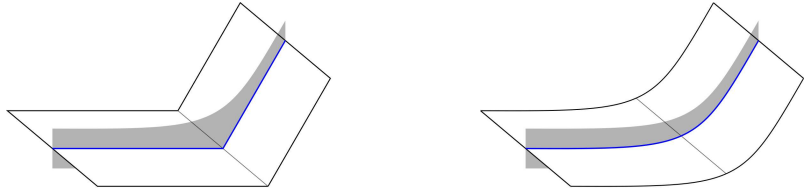}
		\end{overpic}
		\caption{Rounding the corners on $\p Y^{(1)}$ and $\gamma$ (blue). The shaded region represents part of the attaching sphere of the ambient $3$-handle $\widetilde{H}_2^-$ in $\Sigma$.}
		\label{fig:round corner cont'd}
	\end{figure}
	
	\emph{From now on, all the $1$-handles $H_1^{\pm}$ are assumed to be smooth.} 
	
	As it turns out, the $1$-handles $H_1^{\pm}$ which are slices of ambient $2$-handles are too rigid for our later purposes of getting rid of singularities on the $H_2^-$'s.  Hence we will spend the next two substeps on certain modifications of $H_1^+$ and $H_1^-$, respectively, as preparations to subsequent $2$-handle attachments.

	\s\n
	\textsc{Substep 2.3.} \textit{Transform $H_1^+$.}
	
	\s
	From this point on, we will gradually drop the condition (RA3) imposed in Step 1. Let $p_1 \in H_1^+$ be the critical point. \emph{By a $C^0$-small perturbation of $\Sigma$ relative to $Y$ near $p_1$, one can arrange $\ind_{\Sigma}(p_1)=1$.} Note that this is possible since the stable manifold of $p_1$ in $Y$ is assumed to smooth by Substep 2.2. Therefore we are in the situation of \autoref{subsubsec:poitive H_1} (Case 2), where $\gamma=\mu$ is a meridian great circle on $\p B_0^3$. By the discussions in \autoref{subsubsec:poitive H_1}, a small rotation of $\gamma$ makes it everywhere transverse to $\Gcal|_{\p B_0^3}$. In what follows, let's write $\mu$ for the great circle and $\gamma$ for the slightly rotated copy. Also write $H_1^+(\mu)$ and $H_1^+(\gamma)$ for the corresponding $1$-handles. Then the characteristic foliations on $\p H_1^+(\mu)$ and $\p H_1^+(\gamma)$ are shown in \autoref{fig:char foliation modify H1+}. In particular, observe that $\mu$ is tangent to the characteristic foliation while $\gamma$ is transverse. Moreover, the two saddles on $\p_+ H_1^+(\gamma)$ are separated by $\gamma$.
	
	\begin{figure}[ht]
		\begin{overpic}[scale=.4]{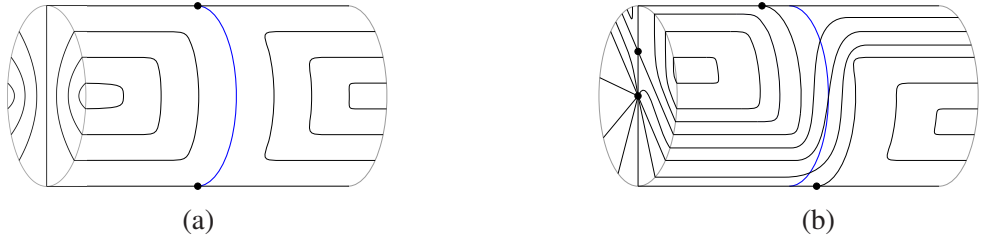}
			\put(18,-4){(a)}
			\put(81.7,-4){(b)}
		\end{overpic}
		\vspace{5mm}
		\caption{(a) The characteristic foliation on $\p H_1^+(\mu)$. (b) The characteristic foliation on $\p H_1^+(\gamma)$. The dots indicate the critical points and the blue circles represent $\mu$ and $\gamma$, respectively.}
		\label{fig:char foliation modify H1+}
	\end{figure}
	
	If we identify $H_1^+(\gamma) \cong B^1(R) \times B^2$ where $B^1(R) \coloneqq [-R,R]$ and $\gamma=\{0\} \times \p B^2$, then for $R'>0$ sufficiently small, the truncated $1$-handle $\widecheck{H}_1^+(\gamma) \coloneqq B^1(R') \times B^2 \subset H_1^+$ is smooth(able) and the Legendrian leaves in $\widecheck{H}_1^+(\gamma)$ are all of the type shown in \autoref{fig:Leg foliation on 1 handle}(d). We call such $\widecheck{H}_1^+(\gamma)$ a \emph{turbine} $1$-handle since it looks like a turbine engine. Note that the characteristic foliation on $\p_+ \widecheck{H}_1^+ (\gamma)$ is linear, and on each (disk) component of $\p_- \widecheck{H}_1^+ (\gamma)$ is a neighborhood of either a source or a sink.
	
	Now the idea is to push $H_1^+(\gamma) \setminus \widecheck{H}_1^+(\gamma)$ down to the $0$-handles. However, for this to work, one must drop the condition (RA3) near $Y$-index $0$ critical point since it imposes too strict restrictions on the characteristic foliations on $\p H_0$ (cf. Substep 2.1). \emph{By a $C^0$-small perturbation of $\Sigma$ relative to $Y$ near the $Y$-index $0$ critical points, one can arrange so that their $\Sigma$-index is also $0$.} The main advantage of this modification is that the characteristic foliation on $\p_+ H_0$ can now be anything realizable by an embedded $S^2 \subset (S^3,\xi_{\std})$. Of course, as a price paid for this extra flexibility, we lose, in general, the smoothness of $H_0$ at the critical point which is turned into a cone singularity. After this modification, one can easily push $H_1^+(\gamma) \setminus \widecheck{H}_1^+(\gamma)$ down to the $0$-handles using the flow of $\Sigma_{\xi}|_{H_1^+(\gamma) \setminus \widecheck{H}_1^+(\gamma)}$. In particular, the saddles on $\p_+ H_1^+(\gamma)$ become saddles on appropriate $\p_+ H_0$'s.
	
	\emph{From now on, all the positive $1$-handles $H_1^+$ are assumed to be turbine $1$-handles.}
	
	\s\n
	\textsc{Substep 2.4.} \textit{Transform $H_1^-$.}
	
	\s
	Let $p_1 \in H_1^-$ be the critical point. Then $\ind_{\Sigma} (p_1)=2$ necessarily, and the trick for $H_1^+$ does not apply to the negative case. Instead, the plan is to create a canceling pair of negative critical points of $Y$-index $1$ and $2$ within $H_1^-$ so that a single negative $1$-handle $H_1^-$ will be turned into a combination of two copies of $H_1^-$ and one $H_2^-$.
	
	To carry out the plan, let's identified the ambient $2$-handle $\widetilde{H}_1^- \cong B^2_x \times B^2_y$ equipped with the \emph{negative} Liouville vector field
	\begin{equation*}
	X_2^- \coloneqq -2x_1 \p_{x_1} - 2x_2 \p_{x_2} + y_1 \p_{y_1} + y_2 \p_{y_2},
	\end{equation*}
	with respect to the standard symplectic form $\omega_{\std}=dx_1 \wedge dy_1 + dx_2 \wedge dy_2$. Recall that near negative critical points, the (oriented) characteristic foliation coincides with the negative Liouville vector field. Then the embedding $H_1^- = B^1 \times B^2 \subset \widetilde{H}_1^-$ is given by $x_2=0$. In particular,
	\begin{equation*}
	X_2^-|_{H_1^-} = -2x_1 \p_{x_1} + y_1 \p_{y_1} + y_2 \p_{y_2}.
	\end{equation*}
	Now one can create a pair of negative critical points $q_1,q_2$ of $Y$-index $1$ and $2$ (and $\Sigma$-index of $2$ and $3$), respectively, along the positive (or negative) $y_2$-axis such that the Legendrian foliation $\Fcal_{H^-_1}$ remains unchanged. See \autoref{fig:modify Leg foliation on negative H1-1}, where we draw only the Legendrian foliation and not explicitly the handles. It follows that the contact germ on $H_1^-$ is unchanged under such modification since it is determined by $\Fcal_{H^-_1}$. To avoid confusions, let's denote the original $H_1^-$ by $H_1^{-,\op{orig}}$, the $1$-handles corresponding to $p_1,q_1$ by $H_1^- (p_1), H_1^- (q_1)$, respectively, and the $2$-handle corresponding to $q_2$ by $H_2^- (q_2)$. Then
	\begin{equation*}
	H_1^{-,\op{orig}} = H_1^- (p_1) \cup H_1^- (q_1) \cup H_2^- (q_2).
	\end{equation*}
	
	\begin{figure}[ht]
		\begin{overpic}[scale=.2]{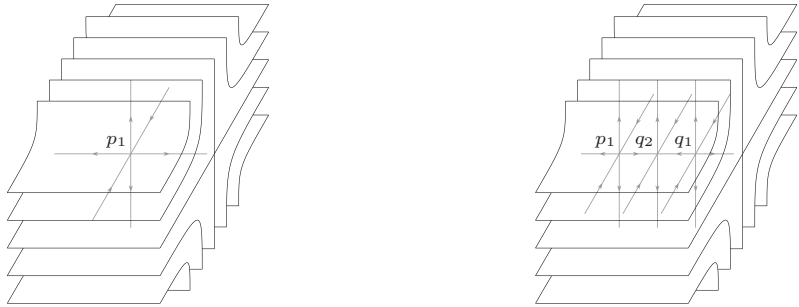}
			\put(12.5,20.3){\tiny{$p_1$}}
			\put(74.5,20.3){\tiny{$p_1$}}
			\put(79.5,20.3){\tiny{$q_2$}}
			\put(84.5,20.3){\tiny{$q_1$}}
		\end{overpic}
		\caption{Creation of critical points $q_1,q_2$ on $H_1^-$ without changing $\Fcal_{H^-_1}$.}
		\label{fig:modify Leg foliation on negative H1-1}
	\end{figure}
	
	Observe that the attaching locus $\p_- H_2^-(q_2)$ contains two saddles $h_{\pm}$ of opposite signs such that the unstable manifold of $h_+$ coincides with the stable manifold of $h_-$. Indeed $\p_- H_2^-(q_2)$ is a tubular neighborhood of the tangential loop $\mu$ passing through $h_{\pm}$. Moreover, up to a flip of orientation, we can assume $h_+ \in \p_+ H_1^-(p_1)$ and $h_- \in \p_+ H_2^-(q_1)$. For later use, let $h_-(p_1) \in \p_+ H_2^-(p_1)$ and $h_+(q_1) \in \p_+ H_2^-(q_1)$ be the other saddles.
	
	Next, we want to perturb $H_2^-(q_2)$ such that $\p_- H_2^-(q_2)$ becomes a neighborhood of a transverse loops instead of a tangential loop. This procedure is dual to the perturbation discussed in Substep 2.3. Namely, the boundary of the stable manifold of the ambient $3$-handle $\widetilde{H}_2^-(q_2)$ is a $2$-sphere equipped with a restricted characteristic foliation identical to $\Gcal|_{\p B_0^3}$. Moreover $\mu$, the boundary of the stable manifold of $q_2$, is a meridian great circle. By the same construction as in Substep 2.3, one can apply a small rotation to $\mu$ to obtain a transverse loop $\gamma$. Then we obtain a perturbed $2$-handle $H_2^{-,\op{pert}} (q_2)$ whose stable manifold is the cone over $\gamma$. See \autoref{fig:modify Leg foliation on negative H1-2}. Note that the left-hand-side of \autoref{fig:modify Leg foliation on negative H1-2} is the same as the right-hand-side of \autoref{fig:modify Leg foliation on negative H1-1}, except that the $2$-handle corresponding to $q_2$ looks somewhat squashed. Let's introduce a piece of notation for later use: a $2$-handle $H_2^-$ is \emph{turbine} if the unstable manifold is a cone over a transverse loop. Note that turbine $2$-handles are dual to turbine $1$-handles introduced in Substep 2.3.
	
	\begin{figure}[ht]
		\begin{overpic}[scale=.18]{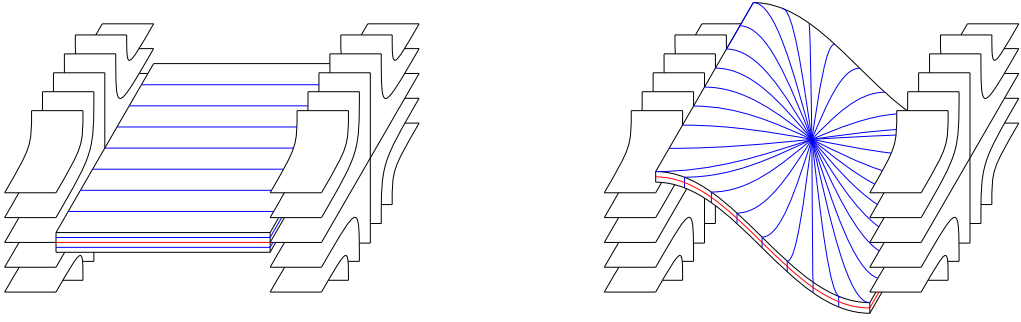}
		\end{overpic}
		\caption{The unperturbed $H_2^-(q_2)$ on the left, and the perturbed $H_2^{-,\op{pert}} (q_2)$ on the right.  The characteristic foliations on $\p H_2^- (q_2)$ are drawn in blue. The red curves represent the tangential $\mu$ on the left and the transverse $\gamma$ on the right.}
		\label{fig:modify Leg foliation on negative H1-2}
	\end{figure}
	
	Write
	\begin{equation*}
	H_1^{-,\op{pert}} \coloneqq H_1^-(p_1) \cup H_1^-(q_1) \cup H_2^{-,\op{pert}}(q_2),
	\end{equation*}
	where $H_2^{-,\op{pert}}(q_2)$ is a turbine $2$-handle. We describe the characteristic foliation on $\p H_1^{-,\op{pert}}$ as follows. On $\p_- H_1^{-,\op{pert}}$, the characteristic foliation is linear, just as in the unperturbed case although the holonomy is slightly changed due to the perturbation. On the other hand, note that the union of the stable manifold of $h_-(p_1)$ and the unstable manifold of $h_+(q_1)$ separates $\p_+ H_1^{-,\op{pert}}$ into two connected components. On one component, one has two critical points $h_+$ and a source in canceling position, and on the other component, one has $h_-$ and a sink also in canceling position. See \autoref{fig:modify Leg foliation on negative H1-2}.
	
	Finally, note that the above transformation of $H_1^-$ can be done repeatedly. The number of times we transform $H_1^-$ as above will depend on how the $2$-handles are attached subsequently.

	\s\n
	\textsc{Substep 2.5.} \textit{Smoothing the $2$-handles $H^-_2$.}
	
	\s
	Let $Y^{(1)}$ be a neighborhood of the $1$-skeleton of $Y$. We introduce a special collection of loops on $\p Y^{(1)}$ as follows: an $\alpha_{\pm}$ curve is the intersection between $\p Y^{(1)}$ and the unstable manifold of a $H_1^{\pm}$, respectively, and a $\beta$ curve is the intersection between $\p Y^{(1)}$ and the stable manifold of a $H_2^-$. By Substep 2.3 and 2.4, $\alpha_+$ curves are transverse and $\alpha_-$ curves are tangential with respect to the characteristic foliation $(\p Y^{(1)})_{\xi}$. By genericity, assume moreover that $\beta$ curves are transverse to $\alpha_{\pm}$ curves. Note, however, that the relative position between a $\beta$ curve and $(\p Y^{(1)})_{\xi}$ can be complicated. The main idea of this substep is, roughly speaking, to ``modify'' the $\beta$ curves to become transverse to $(\p Y^{(1)})_{\xi}$.
	
	Focus on one $\beta$ curve for the moment. We will try to find a parallel copy of $\beta$ which is transverse to $(\p Y^{(1)})_{\xi}$. Clearly this is not possible in general if we keep $(\p Y^{(1)})_{\xi}$ unchanged. First we want to find a parallel copy $\beta'$ of $\beta$ which is \emph{coherently} transverse to $(\p Y^{(1)})_{\xi}$ near the $\alpha_{\pm}$ curves. Here by ``coherently transverse'' we mean the following. Note that $\beta \subset \p Y^{(1)}$ is two-sided. We say $\beta$ is \emph{coherently transverse} to $(\p Y^{(1)})_{\xi}$ in some region (e.g., a neighborhood of the $\alpha_{\pm}$ curves) if $(\p Y^{(1)})_{\xi}$ flows from one chosen side of $\beta$ to the other side in this region. Of course this notion of coherent transversality is stronger than just transversality only when the region in question is disconnected. The coherently transverse $\beta'$ near $\alpha_+$ curves can be produced from $\beta$ by an isotopy near $\alpha_+$ as shown in \autoref{fig:make beta transverse}~(a). On the other hand, the coherently transverse $\beta'$ near $\alpha_-$ curves can be produced by first making a transformation of $H_1^-$ corresponding to the $\alpha_-$ curve, followed by an isotopy of $\beta$. See \autoref{fig:make beta transverse}~(b). Note that the set of $\alpha_-$ curves changes through the transformations of $H_1^-$ as new handles are created.
	
	\begin{figure}[ht]
		\begin{overpic}[scale=.54]{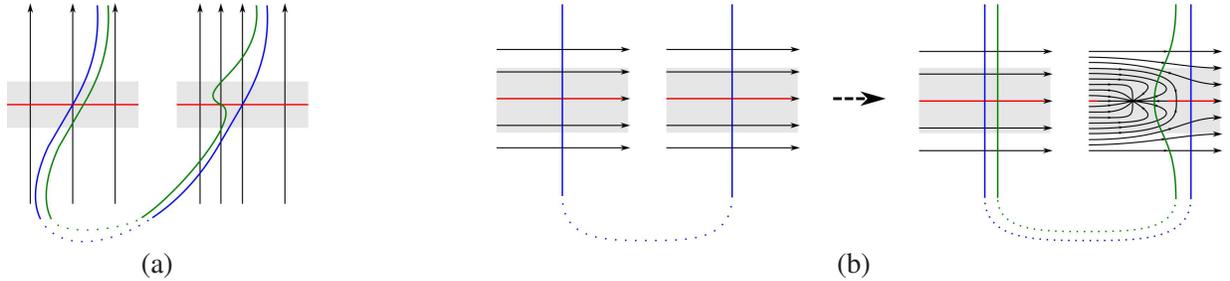}
			\put(11,-2.5){(a)}
			\put(68,-2.5){(b)}
		\end{overpic}
		\vspace{4mm}
		\caption{Making $\beta$ coherently transverse to $\alpha_+$ (left) and $\alpha_-$ (right). The $\alpha$ curves are painted in red, the $\beta$ curves in blue, and the coherently transverse parallel copy of $\beta$ in green. The shaded regions indicate neighborhoods of $\alpha_{\pm}$ curves.}
		\label{fig:make beta transverse}
	\end{figure}
	
	Let $N(\alpha) \subset \p Y^{(1)}$ be a neighborhood of the original $\alpha_{\pm}$ curves, i.e., before the transformations of $H_1^-$. Then the above procedure gives us a parallel copy $\beta'$ of $\beta$ which is coherently transverse to $(\p Y^{(1)})_{\xi}$ within $N(\alpha)$. Now observe that $\p Y^{(1)} \setminus N(\alpha)$ can be viewed as a subsurface of $2$-spheres in $(S^3,\xi_{\std})$. Hence up to a $C^0$-small perturbation of $\p Y^{(1)} \setminus N(\alpha)$ supported in a small neighborhood of $\beta'$, relative to $\beta'$, one can find another parallel copy $\beta''$ which is everywhere transverse to (the perturbed) $(\p Y^{(1)})_{\xi}$. This procedure can be equivalently viewed as a transverse approximation of arcs contained in a surface, which is the $3$-dimensional (trivial) case of the contact approximation described in \cite{HH19}. We apply this transverse approximation to every component of $\beta$.
	
	\emph{From now on, we assume $\beta''$ is a parallel copy of $\beta$ such that each component is transverse.} In particular, $\beta$ and $\beta''$ curves are pairwise disjoint. Let $Y^{(2)}$ be a neighborhood of the $2$-skeleton. Then each component of $\p Y^{(2)}$ can be viewed as a (not necessarily smooth) $2$-sphere in $(S^3,\xi_{\std})$, which is to be filled in by a $3$-handle, i.e., a cone in the standard Liouville $(B^4,\lambda_0)$ (cf. \autoref{subsubsec:H_3} and \autoref{eqn:index 0 Liouville form}). Note that $\beta''$ lies on the smooth part of $\p Y^{(2)}$, i.e., away from the cone points. For simplicity of notation, let's assume both $\beta$ (hence $\beta''$) and $\p Y^{(2)}$ are connected. Otherwise, each connected component can be dealt with separately.
	
	Let $H_3 \subset (B^4,\lambda_0)$ be the filling $3$-handle, i.e., $\p H_3 = \p Y^{(2)}$. Then $\beta''$ can be viewed as a transverse unknot in $(S^3,\xi_{\std})$. Roughly speaking, the idea is to subdivide $H_3$ into two $3$-handles along a $2$-disk bounding $\beta''$ so that the $2$-disk itself can be turned into a (negative) $2$-handle, which will be smooth(able) since $\beta''$ is transverse. Unfortunately, as we will see, this naive idea does not always work. But we will begin with a toy case, where the naive idea does work, under the following additional assumption:
	
	\begin{assump} \label{assump:unknot}
		$\beta'' \subset (S^3,\xi_{\std})$ is the standard transverse unknot (which is the same as the transverse push-off of the standard Legendrian unknot).
	\end{assump}
	
	Recall that a standard $S^2 \subset (S^3,\xi_{\std})$ is a smoothly embedded $2$-sphere whose characteristic foliation has exactly two critical points: a source and a sink. Under \autoref{assump:unknot}, there exists a standard $S^2$ such that $S^2 \cap \p H_3 = \beta''$. Note that $\beta''$, viewed in the standard $S^2$, is a loop winding once around the source (or equivalently, the sink). Now there exists a Liouville homotopy which turns the standard $(B^4,\lambda_0)$ into a collection of two standard balls $B_1,B_2$, joined by a symplectic $1$-handle (cf. \autoref{eqn:index 1 vf} and \autoref{eqn:index 1 Liouville form}) such that the unstable manifold $D \cong B^3$ of the $1$-handle intersects $\p B^4$ along the standard $S^2$ as above. Let $K(\beta'') \subset D$ be the cone over $\beta''$. Then $K(\beta'')$ can be thickened to a coLegendrian (turbine) $2$-handle $H_2^-$. Attach this $H_2^-$ to $Y^{(2)}$ along $\beta''$. Then the boundary of $Y^{(2)} \cup H_2^-$ consists of two $2$-spheres, which can be coned off in $B_1$ and $B_2$, respectively. This completes the ``subdivision'' of $H_3$ as desired. Note, however, that we are not really subdividing $H_3$ since the original $H_3$ does not interact with the Liouville homotopy. Instead, we construct by hand a new filling of $\p Y^{(2)}$ which, as a regular coLegendrian, is build out of one $2$-handle and two $3$-handles.
	
	Now we drop \autoref{assump:unknot} and return to the general case. Since $\beta''$ is topologically the unknot, it is necessarily a (transverse) stabilization of the standard unknot. Indeed, $\beta''$ bounds a $2$-disk $\Delta$ in $(S^3,\xi_{\std})$ such that the characteristic foliation $\Delta_{\xi_{\std}}$ can be normalized to have exactly $m$ sources $e_1,\dots,e_m$ and $m-1$ negative saddles $h_1,\dots,h_{m-1}$, up to a flip of orientation of $\Delta$.\footnote{This is equivalent to saying that the self-linking number of $\beta''$ is $-m$.} Moreover, one can arrange, for definiteness, that each source has at most two saddles connect to it. The plan to to deform $(B^4,\lambda_0)$, not as a Liouville domain as in the above toy model, but rather as a (Morse) hypersurface in a Darboux $5$-ball where critical points of both signs will be created.
	
	More precisely, we will first build a regular coLegendrian $3$-ball in $B^4$, which we still denote by $D$. Thicken $D$ to $D \times [-1,1] \subset B^4$ such that the $[-1,1]$-component of the characteristic foliation points away from $0$. In particular, the $D$-index of every critical point coincides with its $B^4$-index. Finally, we glue two standard $4$-balls to $D \times [-1,1]$ along $D \times \{-1\}$ and $D \times \{1\}$, respectively, to complete the construction of the deformed $(B^4,\lambda_0)$. However, in this case $D$ will be built out of $3m-2$ coLegendrian handles, instead of being the unstable manifold of a single handle as in the toy model, which can be thought of as the case $m=1$. Another technical (and slightly unfortunate) remark is that the build of $D$ will involve handles of $D$-index $3$ but no handles of $D$-index $0$. \emph{Hence it will be more convenient to start from the highest index handles and attach to them the lower index ones, i.e., the attaching locus will be the positive boundaries instead of the usual negative boundaries.}
	
	Here is the recipe to build $D$. First lay out $k$ $D$-index $3$ handles $H_3^1,\dots,H_3^m$. Note that each $H_3^i$ comes with a Legendrian foliation by half-disks bound together along a diameter (cf. \autoref{eqn:Leg foliation on H0}). In particular, the induced characteristic foliation on each $\p H_3^i \cong S^2$, denoted by $\Gcal$, is standard. Define a \emph{square} $\square \subset (S^2,\Gcal)$ to be an embedded rectangle such that $\Gcal|_{\square}$ is a linear foliation parallel to one of the sides. Fix two disjoint squares $\square^i_{\pm} \subset \p H_3^i$ for $1 \leq i \leq m$. Next, attach $m-1$ positive $D$-index $2$ handles $H_2^{1,+},\dots,H_2^{m-1,+}$ such that $H_2^{j,+}$ joins $H_3^{j}$ and $H_3^{j+1}$ with the attaching locus $\p_+ H_2^{j,+}$ identified with $\square^j_+ \cup \square^{j+1}_+$ for $1 \leq j \leq m-1$ (cf. \autoref{subsubsec:positive H_2}). The $m-1$ negative $D$-index $2$ handles are attached similarly. Finally, attach $m-1$ (positive) $D$-index $1$ handles $H_1^{1,+},\dots,H_1^{m-1,+}$ to, intuitively speaking, fill in the holes created by the $2$-handle attachments. Namely, each $H_1^{k,+}, 1 \leq k \leq m-1$, is attached along a transverse loop that traverses once around the negative boundaries of $H_3^k, H_2^{k,+},H_3^{k+1}$, and $H_2^{k,-}$. See \autoref{fig:deform 3-handle}. Note that this particular pattern of the arrangement of the handles is in line with the characteristic foliation on $\Delta$.
	
	\begin{figure}[ht]
		\begin{overpic}[scale=.4]{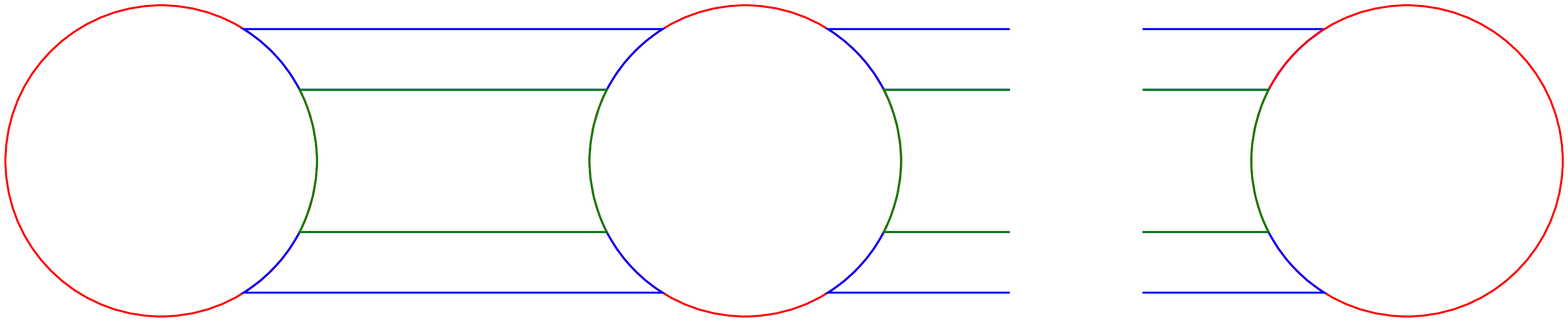}
			\put(8,9){$H_3^1$}
			\put(45.5,9){$H_3^2$}
			\put(88,9){$H_3^m$}
			\put(26,9){\small $H_1^{1,+}$}
			\put(26.5,15.85){\tiny $H_2^{1,+}$}
			\put(26.5,2.8){\tiny $H_2^{1,-}$}
			\put(66.7,16.5){$\dots$}
			\put(66.7,10){$\dots$}
			\put(66.7,3.5){$\dots$}
		\end{overpic}
		\caption{A schematic picture of the construction of $D$ using coLegendrian handles.}
		\label{fig:deform 3-handle}
	\end{figure}
	
	Continuing using the convention from \autoref{subsec:coLeg handles}, the ambient (symplectic) handle in $B^4$ corresponding to a coLegendrian handle $H_{\ast}^{\ast}$ in $D$ is denoted by $\widetilde{H}_{\ast}^{\ast}$. Then we can thicken $D$ as follows
	\begin{equation*}
	D \times [-1,1] \cong \left( \cup_{1 \leq i \leq m} \widetilde{H}_3^i \right) \cup \left( \cup_{1 \leq j \leq m-1} (\widetilde{H}_2^{j,+} \cup \widetilde{H}_2^{j,-} \cup \widetilde{H}_1^{j,+}) \right).
	\end{equation*}
	Note that $\p_+ (D \times [-1,1]) = D \times \{-1,1\}$ is contactomorphic to two copies of Darboux $3$-balls since the handles $\widetilde{H}_1^{j,+}$ and $\widetilde{H}_2^{j,+}$ cancel in pairs. It follows that one can attach two standard Liouville $4$-balls $B_1,B_2$ to $D \times [-1,1]$ along $\p_+ (D \times [-1,1])$ to complete our construction of the perturbed $B^{4,\op{pert}}$. Note that $B^{4,\op{pert}}$, viewed as a hypersurface in the Darboux $5$-ball, is deformation equivalent to the standard $(B^4,\lambda_0)$, relative to $\p B^4$, by the obvious cancellations of handles, e.g., $\widetilde{H}_1^{j,+}$ cancels $\widetilde{H}_2^{j,+}$, and $\widetilde{H}_2^{i,-}$ cancels $\widetilde{H}_3^i$, etc.
	
	The point of the above construction is that there exists a properly embedded regular $2$-disk $\Delta' \subset D$, i.e., $\Delta'$ is tangent to the Morse vector field on $D$, such that the intersection of the Legendrian foliation $\Fcal_D$ with $\Delta'$ gives a vector field identical to $\Delta_{\xi_{\std}}$ as described above. In order to find such $\Delta'$, for each $1 \leq j \leq m-1$, let's identify $H_2^{j,+} \cong B^2_{x_1,x_2} \times B^1_{y_1}$ such that the Legendrian foliation
	\begin{equation*}
	\Fcal_{H^{j,+}_2} = \ker(x_1dy_1-2y_1dx_1).
	\end{equation*}
	Let $S_j \coloneqq \{x_2=0\} \subset H_2^{j,+}$ be a strip. Assume w.l.o.g. that $S_j$ is disjoint from the attaching locus of $H_1^{j,+}$. For each $1 \leq i \leq m$, let $D_i \subset H_3^i$ be a properly embedded disk transverse to $\Fcal_{H_3^i}$ such that
	\begin{itemize}
		\item $S_j \cap \p H_3^j \subset \p D_j$ and $S_j \cap \p H_3^{j+1} \subset \p D_{j+1}$;
		\item $\p D_i$ is disjoint from the attaching loci of $H_1^{\ast,+}$ and $H_2^{\ast,-}$, where $\ast = i$ or $i-1$.
	\end{itemize}
	It follows that the union of all the $D_i$'s and $S_j$'s gives the desired $\Delta'$. Indeed, each $D_i$ contains an $e_i$ and each $S_j$ contains an $h_j$. Now the subdivision of $H_3$ can be done in the way as in the toy case. Namely, one first attach the thickened $\Delta'$ to $\p Y^{(2)}$, and then cone off the two boundary spheres in $B_1$ and $B_2$, respectively.
	
	Finally, we get rid of the original (non-smooth) $2$-handle along $\beta$ by canceling it with (either) one of the adjacent $3$-handles. Applying the above procedure to every component of $\beta$, we replace all the potentially non-smooth $H_2^-$'s by smooth ones, i.e., the attaching locus being a transverse loop, at the cost of introducing extra $H_1^+$'s.
	
	Observe that after all the above steps, the resulting $Y$ can possibly be singular only at critical points of $Y$-index $0$ and $3$, which are precisely the isolated cone singularities. The proof is therefore complete.
\end{proof}

\begin{remark} \label{rmk:necessary coLeg pieces}
	The proof of \autoref{prop:coLeg approx} actually provides more information about the approximating regular coLegendrian $Y$ than what is stated in the Proposition. Namely, one can always build such $Y$ using only the following pieces:
	\begin{itemize}
		\item[(HD1)] Cones over $S^2 \subset (S^3,\xi_{\std})$, which can have $Y$-index $0$ or $3$.
		\item[(HD2)] Positive turbine $1$-handles (cf. Substep 2.3).
		\item[(HD3)] Negative $1$-handles as described in \autoref{subsubsec:negative H_1}.
		\item[(HD4)] Negative turbine $2$-handles (cf. Substep 2.4).
	\end{itemize}
\end{remark}

\subsection{Regular coLegendrians with boundary} \label{subsec:coLeg with bdry}
Suppose $Y \subset (M^5,\xi)$ is a compact $3$-submanifold with Legendrian boundary such that the normal bundle $T_Y M$ is trivial. We say $Y$ is a \emph{regular coLengdrian with boundary} if there exists a Morse hypersurface $\Sigma \supset Y$ such that $\p Y$ is a regular Legendrian with respect to $\Sigma_{\xi}$ and $Y$ is regular as in the closed case. Sometimes we also simply say $Y$ is regular if there is no risk of confusion.

The goal of this subsection is to prove the following relative analog of \autoref{prop:coLeg approx}.

\begin{prop} \label{prop:coLeg approx with bdry}
	Suppose $Y \subset (M^5,\xi)$ is a compact $3$-submanifold with smooth Legendrian boundary such that $T_Y M$ is trivial. Then $Y$ can be $C^0$-approximated, relative to $\p Y$, by a regular coLegendrian with isolated cone singularities in the interior.
\end{prop}

\begin{proof}
	The proof is divided into two steps. The first step is to normalize a collar neighborhood of $\p Y$ in $Y$ such that $Y$ becomes regular near the boundary. Once this is done, the second step, which is to perturb the interior of $Y$, works in the same way as in the closed case.
	
	\s\n
	\textsc{Step 1.} \textit{Build a collar neighborhood of $\p Y \subset Y$ by partial coLegendrian handles.}
	
	\s
	Recall that a closed regular coLegendrian can be built out of coLegendrian handles described in \autoref{subsec:coLeg handles}. In fact, according to \autoref{rmk:necessary coLeg pieces}, only a sub-collection of handles are needed. In the relative case, one needs, in addition, \emph{partial handles} which can be obtained by cutting a coLegendrian handle in half along a (Legendrian) leaf passing through the critical point. Instead of exhausting all possible partial handles, we focus on describing those which we will need to construct the collar neighborhood of $\p Y$. The following list is in line with (HD1)--(HD3) in \autoref{rmk:necessary coLeg pieces}.
	\begin{itemize}
		\item[(PHD1)] Partial handles $PH_0$ and $PH_3$ of index $0$ and $3$, respectively, are modeled on cones over a $2$-disk $D \subset (S^3,\xi_{\std})$ such that $\p D$ is the standard Legendrian unknot.
		\item[(PHD2)] Model a positive turbine $1$-handle $H_1^+ \cong B^1_x \times B^2_{y,z}$ such that the Legendrian foliation $$\Fcal_{H_1^+}=\ker(ydz-zdy).$$ Then a positive turbine partial $1$-handle $PH_1^+$ is modeled on $H_1^+ \cap \{z \geq 0\}$.
		\item[(PHD3)] Model a negative $1$-handle $H_1^- \cong B^1_x \times B^2_{y,z}$ such that the Legendrian foliation  
		\begin{equation} \label{eqn:Leg foliation on negative partial H1}
		\Fcal_{H_1^-} = \ker(xdy+2ydx).
		\end{equation}
		Then the negative partial $1$-handle $PH_1^-$ is modeled on $H_1^- \cap \{y \geq 0\}$.
	\end{itemize}
	Note that the boundary of any partial handle naturally splits into two pieces: the \emph{tangential boundary} which is nothing but the Legendrian leaf along which the coLegendrian is cut apart, and the \emph{transverse boundary} which is transverse to the (restricted) characteristic foliation.  In what follows, we denote the tangential boundary of a partial handle $PH_{\ast}^{\ast}$ by $\p_L PH_{\ast}^{\ast}$, and the transverse boundary by $\p_{\tau} PH_{\ast}^{\ast}$.
	
	Using the partial handles described above, let's construct a regular collar neighborhood of $\p Y$ as follows. Suppose $\p Y$ is a (closed) genus $g$ surface. Note that $\p Y \subset \Sigma$ has self-intersection number equal to zero. It follows from \autoref{lem:Leg self intersection number} that $\p Y$ is balanced, i.e., $\chi(R_+(\p Y)) = \chi(R_-(\p Y))$. Hence we will build a collar neighborhood of $\p Y$ using one $PH_0$, $g$ copies of $PH_1^{\pm}$, respectively, and one $PH_3$, such that the tangential boundaries of these handles glue together to a (balanced) regular Legendrian structure on $\p Y$. 
	
	To spell out more details, let $\p _{\tau}^c PH_{\ast}^{\ast}$ be the intersection of $\p_{\tau} PH_{\ast}^{\ast}$ with the collar neighborhood of $\p Y$. First, identify $\p_{\tau}^c PH_0 = \Rbb/(6g-2)\Zbb \times [0,1]$ such that $\Rbb/(6g-2)\Zbb \times \{0\} =  \p_{\tau} PH_0 \cap \p_L PH_0$. One can arrange the characteristic foliation $(\p_{\tau}^c PH_0)_{\xi}$ on $\p_{\tau}^c PH_0$ such that the following hold.
	\begin{itemize}
		\item The critical points of $(\p_{\tau}^c PH_0)_{\xi}$ are the following:
		\begin{itemize}
			\item half sources at $(4i,0)$ and half sinks at $(4i+2,0)$ for $0 \leq i \leq g-1$.
			\item positive half saddles at $(4g+2j+1,0)$ and negative saddles at $(4g+2j,0)$ for $0 \leq j \leq g-1$.
		\end{itemize}
		\item The stable manifolds of the positive half saddles and the unstable manifold of the negative half saddles are contained in $\Rbb/(6g-2)\Zbb \times \{0\}$.
	\end{itemize}
	In particular $\p_{\tau}^c PH_0$ is not Morse$^+$ due to the presence of flow lines from negative half saddles to positive ones. Nevertheless, one can compute using \autoref{eqn:compute 1-framing} that $\tb(\p_{\tau} PH_0 \cap \p_L PH_0)=-1$, i.e., the link of the $\p Y$-index $0$ critical point is the standard Legendrian unknot. See \autoref{fig:Leg foliation near boundary}.
	
	Next, observe that the characteristic foliation on $\p_{\tau} PH_1^+$ consists of a half source, a half sink, and all flow lines travel from the half source to the half sink. Moreover, the attaching region $\p_- PH_1^+ \cap \p_{\tau} PH_1^+$ has two components, which are neighborhoods of the half source and the half sink, respectively. It follows that the $(i+1)$-th $PH_1^+$ can be attached to $PH_0$ such that the half source matches with the half source on $\p_{\tau}^c PH_0$ at $(4i,0)$, and the half sink matches with the half sink at $(4i+2,0)$, for $0 \leq i \leq g-1$. Similarly for negative $1$-handles, observe that the characteristic foliation on $\p_- PH_1^- \cap \p_{\tau} PH_1^-$ is a square in the sense of Substep 2.5 in the proof of \autoref{prop:coLeg approx}, i.e., all leaves are parallel to one of the sides. It follows that the $(i+1)$-th $PH_1^-$ can be attached to $PH_0$ along squares on $\p_{\tau}^c PH_0$ near $(4i+1,0)$ and $(4i+3,0)$ for $0 \leq i \leq g-1$. See \autoref{fig:Leg foliation near boundary}.
	
	\begin{figure}[ht]
		\begin{overpic}[scale=.3]{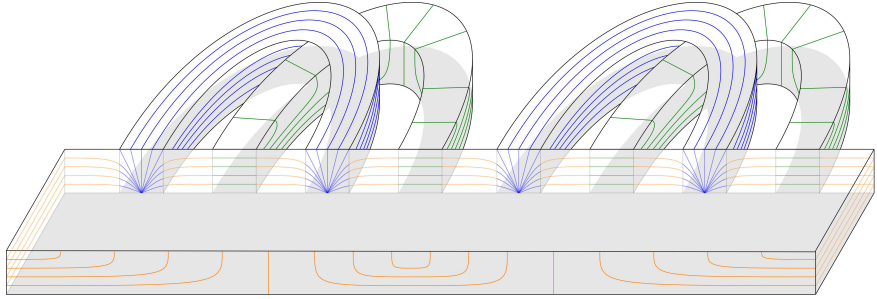}
		\end{overpic}
		\caption{A collar neighborhood of $\p Y$ in the case $g=2$. The bottom sheet (gray) represents $\p Y$, at least away from the $\p Y$-index $2$ handle. The characteristic foliations on $\p_{\tau}^c PH_0$ is drawn in orange (except on overlaps); on $\p_{\tau}^c PH_1^+$'s is drawn in blue; and on $\p_{\tau}^c PH_1^-$'s is drawn in green.}
		\label{fig:Leg foliation near boundary}
	\end{figure}
	
	After attaching the partial $1$-handles above, it remains to cap it off with the last partial $3$-handle $PH_3$ such that the characteristic foliation on $\p_{\tau}^c PH_3$ matches with the one on
	\begin{equation*}
	\p_{\tau}^c \left( PH_0 \cup (\cup_{1 \leq i \leq g-1} PH_1^+) \cup (\cup_{1 \leq i \leq g-1} PH_1^-) \right).
	\end{equation*}
	Moreover, one can check (using \autoref{eqn:compute 1-framing}) that the link of the $\p Y$-index $2$ critical point (which has $Y$-index $3$)  is the standard Legendrian unknot, as expected. Indeed, the characteristic foliation on $\p_{\tau}^c PH_3$ has precisely $2g-1$ positive half saddles and $2g-1$ negative saddles, all of which lie on $\p_{\tau} PH_3 \cap \p_L PH_3$, i.e., the link of the $\p Y$-index $2$ critical point.
	
	\s\n
	\textsc{Step 2.} \textit{Modify the collar neighborhood of $\p Y$ to achieve transversal inner boundary condition.}
	
	\s
	Identified the regular collar neighborhood of $\p Y$ constructed above with $\p Y \times [0,1]$ such that $\p Y$ is identified wit $\p Y \times \{0\}$. A problem of the above construction is that the restricted characteristic foliation $(\p Y \times [0,1])_{\xi}$ is \emph{not} transverse to the \emph{inner} boundary $\p Y \times \{1\}$. Indeed, $(\p Y \times [0,1])_{\xi}$ is pointing away from $\p Y$ near the $\p Y$-index $0$ and $1$ critical points, but is pointing towards $\p Y$ near the $\p Y$-index $2$ critical point. The goal of this step is to modified $(\p Y \times [0,1])_{\xi}$ by adjusting the partial handles such that $\p Y \times \{1\}$ becomes transversal.
	
	At the end of Step 1, we saw that the characteristic foliation on $\p_{\tau}^c PH_3$ has $4g-2$ saddles, which we label by $\{ a^1_{\pm},\dots,a^g_{\pm},b^1_{\pm},\dots,b^{g-1}_{\pm} \}$ such that $a^i_{\pm}$ are the two saddles on the boundary of the $i$-th $PH_1^-$, and the $b^i_{\pm}$'s lie on $\p_{\tau}^c PH_0$. By a $C^0$-small wiggling of $\p_{\tau}^c PH_0$, viewed as an annulus in $(S^3,\xi_{\std})$, relative to a small collar neighborhood of $\p_{\tau}^c PH_0 \cap \p_L PH_0$, one can create (canceling) pairs of critical points above the $b^i_{\pm}$'s as shown in \autoref{fig:perturb char foliation on side}.
	
	\begin{figure}[ht]
		\begin{overpic}[scale=.6]{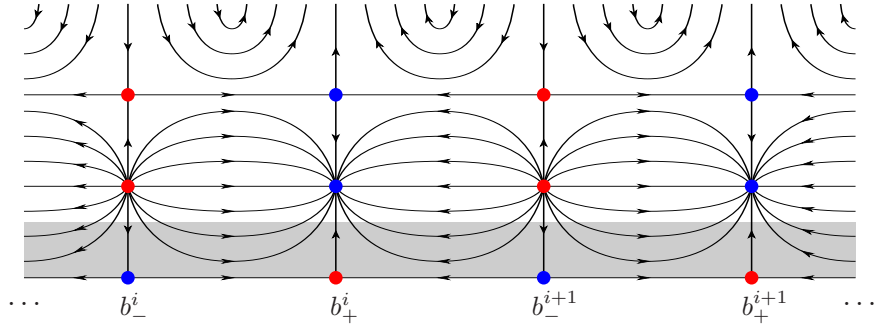}
			\put(11.5,-3.5){\small $b^i_-$}
			\put(36.8,-3.5){\small $b^i_+$}
			\put(61,-3.5){\small $b^{i+1}_-$}
			\put(86.2,-3.5){\small $b^{i+1}_+$}
			\put(-2,-2.3){$\dots$}
			\put(98.3,-2.3){$\dots$}
		\end{overpic}
		\vspace{3mm}
		\caption{Modify the characteristic foliation on $\p_{\tau}^c PH_0$. Positive critical points are marked in red and the negative ones are marked in blue. The gray band represents the small collar neighborhood of $\p_{\tau}^c PH_0 \cap \p_L PH_0$, away from which the perturbation is supported.}
		\label{fig:perturb char foliation on side}
	\end{figure}
	
	We claim that a similar pattern of the creation of critical points can be achieved on each $\p_{\tau}^c PH_1^-$ as well. Indeed, following (PHD3), let's identify $PH_1^- = H_1^- \cap \{y \geq 0\}$. Observe that the ambient contact form restricts to a Liouville form on the $xy$-plane (cf. \autoref{eqn:Leg foliation on negative partial H1}) such that the origin becomes a saddle. Up to a Liouville homotopy, one can create a pair of canceling critical points, one source $e$ and one saddle $h$, on the positive $y$-axis. For definiteness, suppose $e=(0,\tfrac{1}{3},0)$ and $h=(0,\tfrac{2}{3},0)$. As a consequence, $PH_1^-$ splits into $3$ handles, which we call \emph{top, middle} and \emph{bottom}, ordered by the $y$-coordinate, as follows
	\begin{equation*}
	PH_1^{t,-} \coloneqq PH_1^- \cap \{ \tfrac{1}{2} \leq y \leq 1 \},~ PH_1^{m,-} \coloneqq PH_1^- \cap \{ \tfrac{1}{4} \leq y \leq \tfrac{1}{2} \},\text{ and } PH_1^{b,-} \coloneqq PH_1^- \cap \{ 0 \leq y \leq \tfrac{1}{4} \}.
	\end{equation*}
	It follows that $PH_1^{t,-}$ is a negative $1$-handle as in (HD3), $PH_1^{m,-}$ is a (negative) turbine $2$-handle, and $PH_1^{b-}$ is a negative partial $1$-handle, which is isomorphic to the original $PH_1^-$. After applying this modification to all $PH_1^-$'s, the resulting characteristic foliation on $\p_{\tau}^c PH_3$, away from the strip in $\p_{\tau}^c PH_0$ which contains all the $b^i_{\pm}$'s, looks exactly like the one in \autoref{fig:perturb char foliation on side} except for the signs of the critical points. Namely, one needs to put $a^i_{\pm}$ in the places of $b^i_{\mp}$, respectively, and arrange so that the critical points which are aligned vertically have the same sign. In \autoref{fig:perturb char foliation on side}, this means that the dots along each vertical line have the same color, which can be determined by the condition that a source is always red and a sink is always blue.
	
	The point of the above modifications is that now one can easily find a parallel copy $\gamma \subset \p_{\tau}^c PH_3$ of the Legendrian unknot $\p_{\tau}^c PH_3 \cap \p_L PH_3$, which is transverse to the characteristic foliation. The rest of the argument is essentially identical to Substep 2.5 in the proof of \autoref{prop:coLeg approx}, where $\gamma$ plays the role of $\beta''$ there. Namely, one divides $PH_3$ along a $2$-disk bounding $\gamma$ into a partial $3$-handle and a (full) $3$-handle, which we call $H_3$. The $2$-disk itself is turned into a number (depending on $\op{sl}(\gamma)$) of $1$ and $2$-handles.
	
	Finally, let's explain how to make the inner boundary $\p Y \times \{1\}$ transverse to the modified $(\p Y \times [0,1])_{\xi}$. Since all the above perturbations are supported in the interior of $\p Y \times [0,1]$, the relative positions between $\p Y \times \{1\}$ and $(\p Y \times [0,1])_{\xi}$ has not changed at all. Namely, the characteristic foliation is outward pointing along the $0$ and $1$-handles on $\p Y \times \{1\}$ and inward pointing on a $2$-disk $D \subset \p Y \times \{1\}$ representing the $2$-handle. It suffices to ``turn $(\p Y \times [0,1])_{\xi}$ inside out'' along $D$. Observe that $D \subset \p H_3$ is part of a transverse $2$-sphere. Hence the following surface
	\begin{equation*}
	(\p Y \times \{1\} \setminus D) \cup (\p H_3 \setminus D),
	\end{equation*}
	up to corner rounding, is the desired transverse parallel copy of $\p Y$.
	
	\s\n
	\textsc{Step 3.} \textit{Make the interior of $Y$ regular.}
	
	\s
	By the previous steps, we can assume that a collar neighborhood $\p Y \times [0,1]$ of $\p Y = \p Y \times \{0\}$ is regular and $\p Y \times \{1\}$ is transverse to the (restricted) characteristic foliation. In particular $\p Y \times \{1\}$ can be viewed as a surface inside a contact $3$-manifold. Furthermore assume by genericity that the characteristic foliation on $\p Y \times \{1\}$ is Morse$^+$. Then the techniques for making (the interior of) $Y$ regular is the same as in the absolute case (i.e., Step 1 of the proof of \autoref{prop:coLeg approx}) except that one needs \emph{relative C-folds} rather than the usual C-folds in this case as explained in \cite[Section 12]{HH19}.
\end{proof}

\begin{remark}
With a bit more effort, \autoref{prop:coLeg approx with bdry} can be strengthened so that one can prescribe a balanced regular Legendrian structure on $\p Y$, and the coLegendrian approximation can be made relative to a collar neighborhood of $\p Y$.
\end{remark}

\begin{remark}
In a different direction, \autoref{prop:coLeg approx with bdry} can be generalized to allow $\p Y$ to be a singular Legendrian with cone singularities. Cone singularities on a Legendrian $\Lambda$ is defined in a way similar to those on a coLegendrian. Namely, suppose $p_0 \in \Lambda \subset \Sigma$ is a critical point of $\Sigma$-index $0$. Then the link $\link(\Lambda,p_0)$ is a Legendrian unknot in $(S^3,\xi_{\std})$. We say $p_0$ is a \emph{cone singularity} of $\Lambda$. Clearly $\Lambda$ is smooth(able) at $p_0$ if and only if $\tb(\link(\Lambda,p_0))=-1$. For example, if $\tb(\link(\Lambda,p_0))=-2$, then $\Lambda$ is singular at $p_0$, and the singularity is traditionally called a \emph{unfurled swallowtail} (cf. \cite{HW73,EM09}). Note, once again, that the usual unfurled swallowtail refers to the singularity of a (smooth) map, and what we mean here is rather the image of such a map.
\end{remark}

\section{Orange II and overtwistedness} \label{sec:orange II}
In \cite{HH18}, we constructed a singular, but contractible, coLegendrian which we call the \emph{overtwisted orange} $\Ocal$ such that if a contact manifold contains an embedded $\Ocal$, then it is overtwisted. Other coLegendrians which cause overtwistedness include plastikstufe \cite{Nie06} and bLob \cite{MNW13}, at least in dimension $5$ (cf. \cite{Hua17}). In this section, we add yet another overtwisted coLegendrian to the list: the \emph{orange II} $\Ocal_2$, whose main features are the following:
\begin{itemize}
	\item $\Ocal_2$ is diffeomorphic to a ball.
	\item $\Ocal_2$ can be found in some overtwisted Morse (but \emph{not} Morse$^+$) hypersurfaces considered in \cite{HH18}.
	\item The proof that $\Ocal_2$ implies overtwistedness is an immediate consequence of \cite{CMP19}, in contrast to all the other models: plastikstufe, bLob and $\Ocal$, where quite some effort is needed.
\end{itemize}

We continue assuming $\dim M=5$ for the following reason: although, as we will see, the construction of $\Ocal_2$ (and the fact that it implies overtwistedness) can be easily generalized to any dimension, we prefer to view it as a regular coLegendrian but the foundation of regular coLegendrians, i.e., the higher dimensional form of \autoref{thm:coLeg approx}, has not been fully established yet.

As a regular coLegendrian, one builds an $\Ocal_2$ by assembling four pieces together: one $0$-handle $H_0$; one $3$-handle $H_3$; one positive half-handle $PH_2^+$ of $\Ocal_2$-index $2$; and one \emph{strange} negative half-handle $PH_1^{s,-}$ of $\Ocal_2$-index $1$. Here $PH_2^+$ and $PH_1^{s,-}$ are described explicitly as follows (the indexing is a continuation of the ones in \autoref{subsec:coLeg with bdry}).
\begin{itemize}
	\item[(PHD4)] Model a positive $2$-handle $H_2^+ \cong B^2_{x,y} \times B^1_z$ such that the Legendrian foliation
		\begin{equation*}
		\Fcal_{H_2^+} = \ker(xdz+2zdx).
		\end{equation*}
	Then $PH_2^+$ is identified with $H_2^+ \cap \{ z \geq 0 \}$.
	\item[(PHD5)] Model a negative $1$-handle $H_1^- \cong B^1_x \times B^2_{y,z}$ such that the Legendrian foliation
		\begin{equation*}
		\Fcal_{H_1^-} = \ker(xdy+2ydx).
		\end{equation*}
	Then $PH_1^{s,-}$ is identified with $H_1^- \cap \{ x \geq 0 \}$. This is to be distinguished from the $PH_1^-$ constructed in (PHD3).
\end{itemize}

Now the recipe for assembling $\Ocal_2$ consists of the following four steps:
\begin{enumerate}
	\item Prepare a standard $H_0$ in the sense that $\p H_0$ is a standard $2$-sphere in $(S^3,\xi_{\std})$.
	\item Attach $PH_1^{s,-}$ to $H_0$ along a square.
	\item Attach $PH_2^+$ to $PH_1^{s,-}$ such that the stable $2$-disk of $PH_2^+$ is glued to the unstable $2$-disk of $PH_1^{s,-}$ to form a (Legendrian) $2$-sphere, which will be $\p \Ocal_2$.
	\item Cap off the boundary component which is \emph{not} $\p \Ocal_2$ by filling in the standard $H_3$.
\end{enumerate}
In particular, both $H_0$ and $H_3$ are smooth, and so is $\Ocal_2$. Note that the characteristic foliation on $\Ocal_2$ is not Morse$^+$ since when restricted to $\p \Ocal_2$, all the flow lines go from the \emph{negative} source to the \emph{positive} sink.

The Legendrian foliation $\Fcal_{\Ocal_2}$ can be visualized as shown in \autoref{fig:O2}. In particular, away from $\p \Ocal_2$, all leaves of $\Fcal_{\Ocal_2}$ are disks except for one leaf which is an annulus. Since the contact germ on $\Ocal_2$ is completely determined by $\Fcal_{\Ocal_2}$, one can forget about the regular coLegendrian structure and only remember the Legendrian foliation. In this case, one needs to specify the \emph{co-orientations} of the singular loci of $\Fcal_{\Ocal_2}$ as follows. For definiteness, let's identify $\Ocal_2$ with the unit $3$-ball in $\Rbb^3_{x,y,z}$ such that the vertical axis as shown in \autoref{fig:O2} is identified with the $z$-axis. Let $P \coloneqq \{ x \geq 0,y=0 \} \subset \Rbb^3$ be a half-plane. Then $\Ocal_2 \cap P$ is a half-disk equipped with a line field with two critical points, one elliptic $e$ in the interior and one half-hyperbolic $h$ on the boundary. \emph{We require $d\alpha$, restricted to the half-disk, to have opposite signs at $e$ and $h$.} The individual signs are irrelevant, and depend on a choice of the contact form $\alpha$ and the orientation of the half-disk. It is in this form that $\Ocal_2$ can be generalized to arbitrary dimensions in the obvious way.

\begin{figure}[ht]
\begin{overpic}[scale=.3]{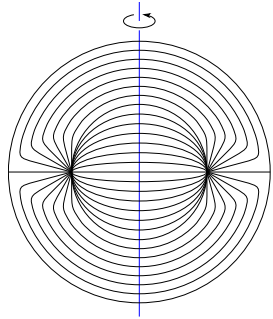}
\end{overpic}
\caption{A cross section of the Orange II, which can be obtained by spinning about the blue axis.}
\label{fig:O2}
\end{figure}

It remains to show that $\Ocal_2$ implies overtwistedness. 

\begin{prop} \label{prop:orange II is OT}
If there exists $\Ocal_2 \subset (M,\xi)$, then $\xi$ is overtwisted.
\end{prop}

\begin{proof}
To facilitate the exposition, let's introduce some notations as follows. Let $A$ be the annulus leaf of $\Fcal_{\Ocal_2}$; $D_1,D_2$ be two disk leaves in the interior; and $N,S \subset \p \Ocal_2$ be the northern and southern hemispheres separated by $\p A \cap \p \Ocal_2$. 

The proof involves considering three Legendrian $2$-spheres. The first one is $\Lambda \coloneqq D_1 \cup D_2$. Throughout the proof we will skip the step of rounding the corners when it is obvious. Clearly $\Lambda$ is the standard Legendrian unknot. The second one is $\Lambda' \coloneqq D_1 \cup A \cup S$, which we claim to be Legendrian isotopic to $\Lambda$. Indeed, it follows from the observation that $D_2$ and $A \cup S$ cobound a $3$-ball in $\Ocal_2$ which is foliated by Legendrian disks with the same boundary $\p D_2$. Finally, the third one is $\p \Ocal_2$ which we claim to be a ``destabilization'' of $\Lambda'$, i.e., $\Lambda'$ is a stabilization of $\p \Ocal_2$. This is most easily seen through the front projection. Namely, on each cross section, $\Lambda'$ is obtained from $\p \Ocal_2$ by a positive and a negative stabilizations. See \autoref{fig:O2 in front}.

\begin{figure}[ht]
\begin{overpic}[scale=.4]{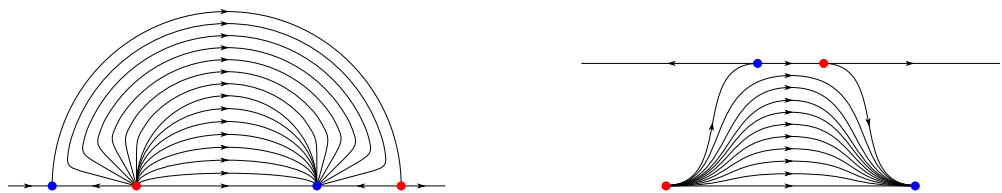}
\end{overpic}
\caption{The front view of a cross section of $\Ocal_2$ between $\Lambda'$ and $\p \Ocal_2$. The red dots indicate positive critical points, while the blue dots indicate negative ones.}
\label{fig:O2 in front}
\end{figure}

To summarize, we have shown that the standard Legendrian unknot $\Lambda$ is a stabilization of $\p \Ocal_2$. The proposition now follows from \cite{CMP19}.
\end{proof}

\begin{remark}
\autoref{prop:orange II is OT} holds in any dimension with the same proof. In particular, in dimension $3$, $\p \Ocal_2$ is a $\tb=1$ Legendrian unknot.
\end{remark}

\bibliography{mybib}
\bibliographystyle{amsalpha}

\end{document}